\documentclass{amsart}

\usepackage{amsmath,
            amssymb,
            bibentry,
            enumitem,
            graphicx,
            mathrsfs,
            tabu,
            tikz,
            xypic,
            wasysym}
\usepackage{xypic}
\usepackage[colorlinks=true,
            linktocpage=true,
            linkcolor=magenta,
            citecolor=magenta,
            urlcolor=magenta]{hyperref}
\PassOptionsToPackage{hyphens}{url}
\usetikzlibrary{arrows,shapes,snakes,positioning}

\newtheorem{theorem}{Theorem}[section]
\newtheorem{proposition}[theorem]{Proposition}
\newtheorem{lemma}[theorem]{Lemma}

\newtheorem{corollary}[theorem]{Corollary}
\theoremstyle{definition}
\newtheorem{definition}[theorem]{Definition}

\newtheorem{question}[theorem]{Question}

\newtheorem{remark}[theorem]{Remark}

\newtheoremstyle{principle}{}{}{\itshape}{}{\bfseries}{.}{.5em}{\thmnote{#3}#1}
\theoremstyle{principle}
\newtheorem*{principle}{}
\newtheoremstyle{case}{}{}{}{}{\itshape}{.}{.5em}{\thmnote{Case #3}#1}

\newcommand{\N}{\naturalnumbers}

\DeclareMathOperator{\res}{\upharpoonright}

\newcommand{\dom}{\operatorname{dom}}

\newcommand{\seq}[1]{\langle #1 \rangle}

\newcommand{\RCA}{\mathsf{RCA}}

\newcommand{\WKL}{\mathsf{WKL}}

\newcommand{\RT}{\mathsf{RT}}
\newcommand{\RTc}{\mathsf{rt}}
\newcommand{\COH}{\mathsf{COH}}

\newcommand{\SRT}{\mathsf{SRT}}

\newcommand{\WWKL}{\mathsf{WWKL}}

\newcommand{\D}{\mathsf{D}}
\newcommand{\LPO}{\mathsf{LPO}}
\newcommand{\LLPO}{\mathsf{LLPO}}
\newcommand{\ACC}{\mathsf{ACC}}
\newcommand{\NON}{\mathsf{NON}}

\newcommand{\CFI}{\mathsf{CFI}}
\newcommand{\K}{\mathsf{K}}

\newcommand{\sub}[1]{_{\textup{\tiny{\fontfamily{cmr}\selectfont #1}}}}

\newcommand{\cequiv}{\equiv\sub{c}}
\newcommand{\cred}{\leq\sub{c}}
\newcommand{\ncred}{\nleq\sub{c}}
\newcommand{\scred}{\leq\sub{sc}}

\newcommand{\scequiv}{\equiv\sub{sc}}
\newcommand{\ured}{\leq\sub{W}}
\newcommand{\uequiv}{\equiv\sub{W}}
\newcommand{\nured}{\nleq\sub{W}}
\newcommand{\sured}{\leq\sub{sW}}
\newcommand{\nsured}{\nleq\sub{sW}}
\newcommand{\suequiv}{\equiv\sub{sW}}

\newcommand{\uint}{{[0, 1]}}

\newcommand{\id}{\mathsf{id}}

\newcommand{\Baire}{\naturalnumbers^\naturalnumbers}
\newcommand{\hide}[1]{}
\newcommand{\mto}{\rightrightarrows}

\newcommand{\C}{\mathsf{C}}
\newcommand{\CC}{\mathsf{CC}}
\newcommand{\lpo}{\mathsf{LPO}}

\newcommand{\Sort}{\mathsf{Sort}}

\newcommand{\fe}[1]{#1^{\textsc{fe}}}
\newcommand{\bal}[1]{\mathsf{b}\text{-}#1}
\newcommand{\unbal}[1]{\mathsf{u}\text{-}#1}
\newcommand{\wub}[1]{\mathsf{wtu}\text{-}#1}
\newcommand{\dwub}[1]{\Delta^0_2\text{-}\mathsf{wtu}\text{-}#1}
\newcommand{\naturalnumbers}{\omega}
\newcommand{\oldoverline}[1]{\widehat{#1}}

\renewcommand{\leq}{\leqslant}
\renewcommand{\geq}{\geqslant}
\renewcommand{\nleq}{\nleqslant}

\renewcommand{\phi}{\varphi}
\newcommand{\converges}{\mathord{\downarrow}}

\newcommand{\comp}{\mathbin{\star}}

\newcommand{\limprob}{\mathsf{lim}}

\begin{document}

\title[Ramsey's theorem and products in the Weihrauch degrees]{Ramsey's theorem and products\\ in the Weihrauch degrees}

\author[Dzhafarov]{Damir D. Dzhafarov}
\address{Department of Mathematics\\
University of Connecticut\\
Storrs, Connecticut, U.S.A.}
\email{damir@math.uconn.edu}

\author[Goh]{Jun Le Goh}
\address{Department of Mathematics\\
Cornell University\\
Ithaca, New York, U.S.A.}
\email{jg878@cornell.edu}

\author[Hirschfeldt]{Denis R. Hirschfeldt}
\address{Department of Mathematics\\
University of Chicago\\
Chicago, Illinois, U.S.A.}
\email{drh@math.uchicago.edu}

\author[Patey]{Ludovic Patey}
\address{Institut Camille Jordan\\
Universit\'e Claude Bernard Lyon 1\\
Lyon, France}
\email{ludovic.patey@computability.fr}

\author[Pauly]{Arno Pauly}
\address{Department of Computer Science\\
Swansea University\\
Swansea, U.K.}
\email{a.m.pauly@swansea.ac.uk}

\thanks{Dzhafarov was supported by grants DMS-1400267 and DMS-1854355 from the National Science Foundation of the United States and a Collaboration Grant for Mathematicians from the Simons Foundation. Goh was supported by NSF grant DMS-1161175. Hirschfeldt was supported by grants
DMS-1101458, DMS-1600543, and DMS-1854279 from the National Science Foundation of the United
States and a Collaboration Grant for Mathematicians from the Simons
Foundation. All authors thank the Leibniz-Zentrum f\"{u}r Informatik at Schloss Dagstuhl, where the initial work for this project was conducted. We also thank Vasco Brattka for asking questions that motivated much of this work, and the anonymous referees for several highly useful comments.}

\begin{abstract}
We study the positions in the Weihrauch lattice of parallel products of various combinatorial principles related to Ramsey's theorem. Among other results, we obtain an answer to a question of Brattka, by showing that Ramsey's theorem for pairs ($\RT^2_2$) is Weihrauch-incomparable to the parallel product of the stable Ramsey's theorem for pairs and the cohesive principle ($\SRT^2_2 \times \COH$).
\end{abstract}

\maketitle

\section{Introduction}\label{s:intro}

Reverse mathematics is a foundational area of logic devoted to calibrating the precise axioms needed to prove a given theorem of ordinary mathematics. For a standard reference, see Simpson~\cite{Simpson-2009}. A particularly fruitful line of research in this endeavor has been looking at theorems from combinatorics, particularly Ramsey's theorem and its many variants. See Hirschfeldt~\cite{Hirschfeldt-2014} for an introduction to the area. One recent way of extending the scope of this analysis is to replace the traditional framework of reverse mathematics, which is provability in fragments of second-order arithmetic, by Weihrauch reducibility. The latter is a tool that has been widely deployed in computable analysis and complexity theory; see the recent survey article by Brattka, Gherardi, and Pauly~\cite{BGP-TA}. Recently it has gained prominence also in the study of computable combinatorics, and it is currently seeing a surge of activity; see, e.g.,~\cite{ADSS-2017, BR-2017, DDHMS-2016, Dzhafarov-2016, DPSW-2017, HJ-2016, HM-2017, HM-2019, MP-TA, Nichols-TA, Patey-2016c, Patey-2016}. See also Brattka~\cite{Brattka-bib} for an updated bibliography.

In this paper, we turn the lens of Weihrauch reducibility on various results concerning Ramsey's theorem and its products with other mathematical principles.
We begin with some background on Weihrauch reducibility and Ramsey's theorem.

\begin{definition}
	A \emph{problem} $\mathsf{P}$ is a partial multifunction from $2^\omega$ to $2^\omega$, written $\mathsf{P} : \mathop{\subseteq} 2^\omega \rightrightarrows 2^\omega$. We call each $X \in \dom(\mathsf{P})$ an \emph{instance of $\mathsf{P}$}, or \emph{$\mathsf{P}$-instance} for short, and each $Y \in \mathsf{P}(X)$ a \emph{solution to $X$ as an instance of $\mathsf{P}$}, or just a \emph{$\mathsf{P}$-solution to $X$}.
\end{definition}

\noindent In general, a problem may be a partial multifunction between other kinds of \emph{represented spaces}. We shall consider such problems in Section \ref{s:CFI}, and refer the reader to~\cite[Section 2]{BGP-TA} for definitions. Elsewhere in this paper, the above definition will be sufficient. (To be precise, we do work with
represented spaces, since we code objects such as colorings of
$n$-tuples of natural numbers as elements of Cantor space, but our
codings are transparent enough that we can safely ignore this
distinction, which we believe will improve clarity for most readers.)

We assume familiarity with standard computability-theoretic notation. For a partial function $\psi$, we write $\psi(x) \simeq y$ to mean that $\psi(x)$ is equal to $y$ if defined.

A broad class of problems comes from reverse mathematics, where a typical object of study is a mathematical principle of the syntactic form
\[
	(\forall X)[\phi(X) \to (\exists Y)[\theta(X,Y)]],
\]
where $\phi$ and $\theta$ are arithmetical formulas of second-order arithmetic. Such a principle gives rise to the problem whose instances are the sets $X$ such that $\phi(X)$ holds, and where the solutions to any such $X$ are the $Y$ such that $\theta(X,Y)$ holds. In general, the formulas $\phi$ and $\theta$ above need not be unique for a given principle, but in practice, each principle one studies has a natural such pair of formulas associated to it. We adopt this terminology for specifying problems in this paper.

\begin{definition}\label{D:reductions}
	Let $\mathsf{P}$ and $\mathsf{Q}$ be problems.
	\begin{enumerate}
		\item $\mathsf{Q}$ is \emph{computably reducible} to $\mathsf{P}$, written $\mathsf{Q} \cred \mathsf{P}$, if every instance $X$ of $\mathsf{Q}$ computes an instance $\widehat{X}$ of $\mathsf{P}$, such that for every solution $\widehat{Y}$ to $\widehat{X}$, we have that $X \oplus \widehat{Y}$ computes a solution $Y$ to $X$.
		\item $\mathsf{Q}$ is \emph{strongly computably reducible} to $\mathsf{P}$, written $\mathsf{Q} \scred \mathsf{P}$, if every instance $X$ of $\mathsf{Q}$ computes an instance $\widehat{X}$ of $\mathsf{P}$, such that every solution $\widehat{Y}$ to $\widehat{X}$ computes a solution $Y$ to $X$.
		\item $\mathsf{Q}$ is \emph{Weihrauch reducible} to $\mathsf{P}$, written $\mathsf{Q} \ured \mathsf{P}$, if there exist Turing functionals $\Phi$ and $\Psi$ such that for every instance $X$ of $\mathsf{Q}$, we have that $\Phi^X$ is an instance of $\mathsf{P}$, and for every solution $\widehat{Y}$ to $\Phi^X$ we have that $\Psi^{X \oplus \widehat{Y}}$ is a solution to $X$.
		\item $\mathsf{Q}$ is \emph{strongly Weihrauch reducible} to $\mathsf{P}$, written $\mathsf{Q} \sured \mathsf{P}$, if there exist Turing functionals $\Phi$ and $\Psi$ such that for every instance $X$ of $\mathsf{Q}$, we have that $\Phi^X$ is an instance of $\mathsf{P}$, and for every solution $\widehat{Y}$ to $\Phi^X$ we have that $\Psi^{\widehat{Y}}$ is a solution to $X$.
	\end{enumerate}
\end{definition}

We write $\mathsf{P} \cequiv \mathsf{Q}$ if $\mathsf{P} \cred \mathsf{Q}$ and $\mathsf{Q} \cred \mathsf{P}$, and similarly for the other reducibilities above. All of these reducibilities are transitive, so the resulting notions of equivalence are in fact equivalence relations, which yield degree structures in the usual way. Figure \ref{F:reductionsrelations} summarizes the relationships that hold between these reducibilities. We refer the reader to Hirschfeldt and Jockusch~\cite[Section 4.1]{HJ-2016} for a more thorough discussion of these reducibilities, and for various generalizations of them with applications to reverse mathematics.

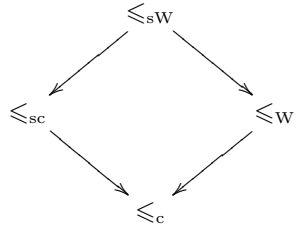
\begin{figure}\label{F:reductionsrelations}
\[
\xymatrix{
& \sured \ar[dl] \ar[dr] \\
\scred \ar[dr] & & \ured \ar[dl] \\
& \cred
}
\]
\caption[]{Relations between notions of reduction. An arrow from one reducibility to another means that whenever $\mathsf{Q}$ is reducible to $\mathsf{P}$ according to the first then it is also reducible according to the second. In general, no relations hold other than the ones shown.}
\end{figure}

The following two definitions list several important operations one can perform on problems.

\begin{definition}
	Let $\mathsf{P}_0$ and $\mathsf{P}_1$ be problems.
	\begin{enumerate}
		\item The \emph{(parallel) product} of $\mathsf{P}_0$ and $\mathsf{P}_1$, written $\mathsf{P}_0 \times \mathsf{P}_1$, is the problem whose instances are pairs $\seq{X_0,X_1}$ with $X_i$ a $\mathsf{P}_i$-instance, and where the solutions to $\seq{X_0,X_1}$ are all pairs $\seq{Y_0,Y_1}$ with $Y_i$ a $\mathsf{P}_i$-solution to $X_i$.
		\item The \emph{coproduct} of $\mathsf{P}_0$ with $\mathsf{P}_1$, written $\mathsf{P_0} \sqcup \mathsf{P_1}$, is the problem whose instances are all pairs $\seq{i,X}$ for $i < 2$ such that $X$ is a $\mathsf{P}_i$-instance, and where the solutions to $\seq{i,X}$ are just the $\mathsf{P}_i$-solutions to $X$.
		\item The \emph{meet} of $\mathsf{P}_0$ with $\mathsf{P}_1$, written $\mathsf{P_0} \sqcap \mathsf{P_1}$, is the problem whose instances are all pairs $\seq{X_0,X_1}$ such that for each $i < 2$, $X_i$ is a $\mathsf{P}_i$-instance, and the solutions to $\seq{X_0,X_1}$ are all pairs $\seq{i,Y}$ for $i < 2$ such that $Y$ is a $\mathsf{P}_i$-solution to $X_i$.
	\end{enumerate}
\end{definition}

It is easy to see that the above operations lift to the $\uequiv$-, $\suequiv$-, $\cequiv$-, and $\scequiv$-degrees. Furthermore, it is easy to see that the $\uequiv$-, $\cequiv$-, and $\scequiv$-degrees form a lattice with $\sqcup$ as join and $\sqcap$ as meet. Recently, Dzhafarov~\cite{Dzhafarov-TA} has shown that the $\suequiv$-degrees also form a lattice, with $\sqcap$ as meet but using a different operation for the join than $\sqcup$.

In this paper, all problems we consider will have some computable instance. It is easy to see that the coproduct of such problems is Weihrauch reducible to their parallel product.

\begin{definition}
	Let $\mathsf{P}_0$ and $\mathsf{P}_1$ be problems.
	\begin{enumerate}
		\item The \emph{composition} of $\mathsf{P}_1$ with $\mathsf{P}_0$, written
$\mathsf{P}_1 \circ \mathsf{P}_0$, is the problem whose instances are
all the $\mathsf{P}_0$-instances $X$ such that every solution to $X$
is a $\mathsf{P}_1$-instance, and whose solutions to such an instance
$X$ are all the $\mathsf{P}_1$-solutions to the
$\mathsf{P}_0$-solutions to $X$.
		\item The \emph{compositional product} of $\mathsf{P}_1$ with $\mathsf{P}_0$, written $\mathsf{P}_1 \comp \mathsf{P}_0$, is defined as $\max_{\ured} \{ \mathsf{Q}_1 \circ \mathsf{Q}_0 : \mathsf{Q}_i \ured \mathsf{P}_i\}$.
	\end{enumerate}
\end{definition}

\noindent The compositional product $\mathsf{P}_1 \comp \mathsf{P}_0$, first defined by Brattka, Gherardi, and Marcone~\cite[Definition 4.1]{BGM-2012}, captures exactly what can be achieved by applying $\mathsf{P}_0$ and $\mathsf{P}_1$ consecutively in series (possibly with some intermediate computation). Brattka and Pauly~\cite{BP-TA} showed that the compositional product is always defined. The definition of $\mathsf{P}_1 \comp \mathsf{P}_0$ above does not yield a specific problem, of course, but only a Weihrauch degree. We will not use this notion except in the context of Weihrauch reducibility, however, so this fact will pose no problems.

The following proposition summarizes the relationships between the above operations on problems.

\begin{proposition} \label{prop:rshp_meet_coproduct_parallel_product_compositional_product}
If $\mathsf{P}_0$ and $\mathsf{P}_1$ are problems with some computable instance, then
\[ \mathsf{P}_0 \sqcap \mathsf{P}_1 \ured \mathsf{P}_0 \sqcup \mathsf{P}_1 \ured \mathsf{P}_0 \times \mathsf{P}_1 \ured \mathsf{P}_0 \comp \mathsf{P}_1. \]
\end{proposition}

To illustrate the definition of $\comp$, we provide a proof of the last reduction. Observe that $\mathsf{P}_0 \times \mathsf{P}_1$ is the same problem as $(\mathsf{P}_0 \times \id) \circ (\id \times \mathsf{P}_1)$. Since $\mathsf{P}_0 \times \id \ured \mathsf{P}_0$ and $\mathsf{P}_1 \times \id \ured \mathsf{P}_1$, we have that $(\mathsf{P}_0 \times \id) \circ (\id \times \mathsf{P}_1) \ured \mathsf{P}_0 \comp \mathsf{P}_1$, completing the proof.

Our focus here will be on Ramsey's theorem and its various combinatorial relatives. We begin with some definitions.

\begin{definition}\label{def:ramsey}
	Let $X$ be a subset of $\omega$ and $k$ a positive number.
	\begin{enumerate}
		\item $[X]^2 = \{ \seq{x,y} \in X \times X : x < y\}$.
		\item A \emph{$k$-coloring of pairs} (frequently just \emph{coloring}) is a function $c : [\omega]^2 \to k$. We write $c(x,y)$ instead of $c(\seq{x,y})$ for $\seq{x,y} \in [X]^2$. The coloring is \emph{stable} if for every $x$ there is an $i < k$ such that $c(x,y) = i$ for all sufficiently large $y$, in which case we write $\lim_y c(x,y) = i$.
		\item A set $H \subseteq X$ is \emph{homogeneous} for such a $c$ if $c \res [H]^2$ is constant. A set $Y \subseteq X$ is \emph{almost homogeneous} for $c$ if there is a finite set $F$ such that $Y - F$ is homogeneous for $c$.
		\item A set $L \subseteq X$ is \emph{limit-homogeneous} for $c$ if there is an $i < k$ such that $c(x,y) = i$ for all $x \in L$ and all sufficiently large $y \in L$, in which case we write $\lim_{y \in L} c(x,y) = i$. A set $Y \subseteq X$ is \emph{almost limit-homogeneous} for $c$ if there is a finite set $F$ such that $Y - F$ is limit-homogeneous for $c$.
	\end{enumerate}	
\end{definition}

\noindent If $i < k$ is the color witnessing that some set is homogeneous or limit-homogeneous then we say the set is homogeneous/limit-homogeneous \emph{with color $i$}. Note that if $c$ is stable and $L$ is limit-homogeneous for $c$ with color $i$ then also $\lim_y c(x,y) = i$ for all $x \in L$.

The following mathematical principles are well-known, and have been studied extensively in computability theory, reverse mathematics, and more recently, in the context of Weihrauch reducibility.
\begin{principle}[Ramsey's theorem for $k$-colorings of pairs ($\RT^2_k$)]
	For every coloring $c : [\omega]^2 \to k$, there is an infinite homogeneous set for $c$.
\end{principle}
\begin{principle}[Stable Ramsey's theorem for $k$-colorings of pairs ($\SRT^2_k$)]
	For every stable coloring $c : [\omega]^2 \to k$, there is an infinite homogeneous set for $c$.
\end{principle}
\begin{principle}[$\Delta^0_2$ $k$-partition subset principle ($\D^2_k$)]
	For every stable coloring $c : [\omega]^2 \to k$, there is an infinite limit-homogeneous set for $c$.
\end{principle}
\noindent (So, for concreteness, the instances of $\RT^2_k$ are all colorings $c : [\omega]^2 \to k$, and the solutions to a given such $c$ are its infinite homogeneous sets. Similarly for the other problems.) One additional principle that has been studied extensively alongside $\RT^2_2$ and $\SRT^2_k$ is the following:
\begin{principle}[Cohesive principle $(\COH$)]
	For every sequence $\seq{c_0,c_1,\ldots}$, where $c_i : \omega \to 2$ for each $i \in \omega$, there exists an infinite set $X$ that is almost homogeneous for each $c_i$.
\end{principle}

It is an easy exercise to see that $\D^2_k$ is strongly Weihrauch equivalent to the problem asserting that for every $\Delta^0_2$ $k$-partition $\seq{A_0,\ldots,A_{k-1}}$ of $\omega$, there exists an infinite subset $X$ of some $A_i$, and in the sequel, we will use whichever formulation is more convenient. It is obvious that $\D^2_k \sured \SRT^2_k \sured \RT^2_k$. While every computable instance of $\SRT^2_k$ has a $\Delta^0_2$ solution, Jockusch~\cite[Theorem 3.1]{Jockusch-1972} constructed a computable instance of $\RT^2_k$ with no $\Delta^0_2$ solution. Thus, $\RT^2_k \nured \SRT^2_k$. Dzhafarov~\cite[Corollary 3.3]{Dzhafarov-2016} showed that $\SRT^2_k \nured \D^2_k$. An independent proof can be found in~\cite[Corollary 6.12]{BR-2017}. Note that if $j < k$ then the version of each of the above principles for $j$-colorings is strongly Weihrauch reducible to the version for $k$-colorings. Patey~\cite{Patey-2016} showed that the converse is false; in fact, if $j < k$, then even $\D^2_k \ncred \RT^2_j$. Further relationships between $\SRT^2_2$, $\D^2_2$, and related principles under the various reductions from Definition \ref{D:reductions} have been investigated by Nichols~\cite{Nichols-TA}.

\begin{definition}
	For a problem $\mathsf{P}$, let $\fe{\mathsf{P}}$ be the problem whose instances are the same as those of $\mathsf{P}$, but such that $Y$ is a $\fe{\mathsf{P}}$-solution to $X$ if there is a $\mathsf{P}$-solution $Z$ to $X$ such that $Y =^* Z$ (i.e., such that $Y$ and $Z$ agree on a cofinite domain).	
\end{definition}

Thus, for instance, $\fe{(\D^2_k)}$ asserts that every stable coloring $c : [\omega]^2 \to k$ has an infinite almost limit-homogeneous set. For some well-behaved principles $\mathsf{P}$, we can express $\fe{\mathsf{P}}$ in terms of the \emph{implication operation} introduced by Brattka and Pauly in~\cite[Section 3.3]{BP-TA}. In lieu of a definition, we use the following property (see~\cite[Theorem 3.13]{BP-TA}): for problems $\mathsf{P}$ and $\mathsf{Q}$, the infimum $\inf_{\ured} \{\mathsf{R} : \mathsf{P} \ured \mathsf{Q} \comp \mathsf{R}\}$ exists and is Weihrauch equivalent to $\mathsf{Q} \to \mathsf{P}$. We also recall the following \emph{choice principle} (see~\cite[Section 7]{BGP-TA}).

\begin{definition}
	$\C_{\mathbb{N}}$ is the problem whose instances are functions $e : \omega^2 \to 2$ such that
	\begin{itemize}
		\item for all $x$, $e(x,0) = 0$ and there is at most one $s$ with $e(x,s) \neq e(x,s+1)$;
		\item there is at least one $x$ with $e(x,s) = 0$ for all $s$.
	\end{itemize}
	A solution to such an $e$ is any $x \in \omega$ such that $e(x,s) = 0$ for all $s$.
\end{definition}
\noindent Thus the instances of $\C_{\mathbb{N}}$ are enumerations of sets with nonempty complements, and the solutions are the elements of these complements.

\begin{proposition}
	Let $\mathsf{P} \in \{\RT^2_k, \SRT^2_k\}$. Then $\fe{\mathsf{P}} \uequiv \mathsf{C}_{\mathbb{N}} \rightarrow \mathsf{P}$.
\end{proposition}

\begin{proof}
	To show that $\fe{\mathsf{P}} \ured \mathsf{C}_{\mathbb{N}} \rightarrow \mathsf{P}$, it suffices to show that for any $\mathsf{R}$ such that $\mathsf{P} \ured \mathsf{C}_{\mathbb{N}} \comp \mathsf{R}$, we have that $\fe{\mathsf{P}} \ured \mathsf{R}$. Equivalently, we will show that for any $\mathsf{Q} \ured \mathsf{C}_{\mathbb{N}}$ and $\mathsf{R}$ such that $\mathsf{P} \ured \mathsf{Q} \circ \mathsf{R}$, we have that $\fe{\mathsf{P}} \ured \mathsf{R}$. Let $\Phi$ and $\Psi$ witness that $\mathsf{P} \ured \mathsf{Q} \circ \mathsf{R}$. Let $\Gamma$ and $\Delta$ witness that $\mathsf{Q} \ured \mathsf{C}_{\mathbb{N}}$. We describe a uniform procedure for reducing $\fe{\mathsf{P}}$ to $\mathsf{R}$. Given a $\mathsf{P}$-instance $c$, we use $\Phi$ to convert this to an instance $X$ of $\mathsf{Q} \circ \mathsf{R}$. Any $\mathsf{R}$-solution $Y$ to $X$ is also a $\mathsf{Q}$-instance, so we can use $\Gamma$ to convert it into an instance $Z$ of $\C_{\mathbb{N}}$. More precisely, $\Gamma^Y$ enumerates a set $\overline{Z}$ such that $Z \neq \emptyset$. And given any $\C_{\mathbb{N}}$-solution to $Z$, i.e., a point $x \in Z$, $\Delta(\seq{Y,x})$ is a $\mathsf{Q}$-solution to $Y$. Hence $\Psi(\seq{c,\Delta(\seq{Y,x})})$ must be a solution to $\mathsf{P}$. Thus, to uniformly compute a $\fe{\mathsf{P}}$-solution $H$ to $c$ from a given $\mathsf{R}$-solution $Y$ to $X$, we proceed as follows. To determine $H(n)$, we choose the least $x$ not yet enumerated by $\Gamma^Y$ at stage $n$, and wait for $x$ either to be enumerated, in which case we let $H(n)=0$, or for $\Psi(\seq{c,\Delta(\seq{Y,x})})(n)$ to converge, in which case we let $H(n) = \Psi(\seq{c,\Delta(\seq{Y,x})})(n)$. It is easy to see that $H$ is then an infinite set and is almost homogeneous for $c$.

	In the other direction, it suffices to show that $\mathsf{P} \ured \mathsf{C}_{\mathbb{N}} \comp \fe{\mathsf{P}}$. Consider the following uniform procedure. Given an instance $c : [\omega]^2 \to k$ of $\mathsf{P}$, we regard it also as an instance of $\fe{\mathsf{P}}$. Now, given any $\fe{\mathsf{P}}$-solution $Y$ to $c$, i.e., an infinite almost homogeneous set, define
	\[
		Z = \{ x \in Y : (\exists i < k)(\forall y \geq x)(\forall z > y)[y,z \in Y \to c(y,z) = i ] \}.
	\]
	Note that $Z$ agrees with $Y$ on all but finitely many elements, and so is in particular nonempty. Moreover, $Z$ is a $\Pi^{0,Y}_1$ subset of $\mathbb{N}$, and hence can be passed as an instance to $\C_{\mathbb{N}}$. Let $x$ be any $\C_{\mathbb{N}}$-solution to this instance. Then $\{y \in Y : y \geq x\}$ is a $\mathsf{P}$-solution to $c$.
	
	We use the above uniform procedure to show that $\mathsf{P} \ured \mathsf{C}_{\mathbb{N}} \comp \fe{\mathsf{P}}$. Let $G$ be a computable function that takes in a pair $\langle c,Y \rangle$ and produces an enumeration of the complement of $Z$, as defined above. Then the above procedure shows that $\mathsf{P} \ured \mathsf{C}_{\mathbb{N}} \circ G \circ (\id \times \fe{\mathsf{P}})$. Since $G \circ (\id \times \fe{\mathsf{P}}) \ured \fe{\mathsf{P}}$, this proves the desired result.
\end{proof}

%

\section{Ramsey's theorem for pairs}

Our starting point is the following summary of known facts concerning relationships between $\RT^2_2$, $\SRT^2_2$, and $\COH$ under Weihrauch reducibility.

\begin{theorem}\label{thm:RT22reds}\
	\begin{enumerate}
		\item $\SRT^2_2 \sqcup \COH \ured \RT^2_2 \ured \SRT^2_2 \comp \COH$;
		\item $\SRT^2_2 \sqcup \COH \ured \SRT^2_2 \times \COH \ured \SRT^2_2 \comp \COH$.
	\end{enumerate}
\end{theorem}

\begin{proof}
  Part (1) follows by the proof of Cholak, Jockusch, and Slaman~\cite[Lemma 7.11]{CJS-2001} that $\SRT^2_2 \wedge \COH$ and $\RT^2_2$ are equivalent in the formal system $\RCA_0$, together with the proof of their Theorem 12.5, which is needed because the argument that $\RT^2_2$ implies $\COH$ in the proof of Lemma 7.11 was not correct, as noted in~\cite{CJS-err} (see also~\cite[Section 5.2]{DDHMS-2016}). Part (2) follows from Proposition \ref{prop:rshp_meet_coproduct_parallel_product_compositional_product}.
\end{proof}

Our main motivation for this section is the following question, asked during the workshop ``Measuring the Complexity of Computational Content: Weihrauch Reducibility and Reverse Analysis'', at the Leibniz-Zentrum f\"{u}r Informatik at Schloss Dagstuhl in September, 2015.

\begin{question}[Brattka, see~\cite{dagstuhl-weihrauch}]\label{Q:main}
	What additional reductions hold between the problems $\SRT^2_2 \sqcup \COH$, $\SRT^2_2 \comp \COH$, $\SRT^2_2 \times \COH$, and $\RT^2_2$ in Theorem \ref{thm:RT22reds}?
\end{question}

We will answer this question by showing that $\RT^2_2$ and $\SRT^2_2 \times \COH$ are Weih\-rauch-incomparable, and hence there are no Weihrauch reductions between the above principles other than the ones given in Theorem \ref{thm:RT22reds}.

We begin by recalling some ancillary notions.

\begin{definition}\label{def:unbalanced}
	Let $c : [\omega]^2 \to k$ be a coloring, and let $X$ be a set.
	\begin{enumerate}
		\item The coloring $c$ is \emph{unbalanced on $X$} if for some $i < k$, every infinite homogeneous set for $c$ contained in $X$ has color $i$. If $c$ is not unbalanced on $X$, it is \emph{balanced on $X$}.
		\item The coloring $c$ \emph{avoids the color $i < k$ on $X$} if $c(x,y) \neq i$ for all $x,y \in X$.
	\end{enumerate}
\end{definition}

\noindent If, in the definition above, $X = \omega$, we shall say simply that $c$ is unbalanced / balanced / avoids the color $i$, without further qualification.

The following lemma will allow us to prove our main result, from which we will derive a number of consequences, including an answer to Question \ref{Q:main}.

\begin{lemma}\label{lem:RT22NON}
	Let $c : [\omega]^2 \to k$ be a computable coloring, $A$ an infinite computable set, and $\mathcal{C} \subseteq 2^{\omega}$ a nonempty $\Pi^0_1$ class of $k$-partitions of $A$. If, for every $\seq{P_0,\ldots,P_{k-1}} \in \mathcal{C}$, $c$ is unbalanced on $P_j$ for every $j < k$, then $c$ has a computable infinite homogeneous set.
\end{lemma}

\begin{proof}
	Fix $c$ and $\mathcal{C}$, and suppose that $c$ has no computable infinite homogeneous set. We construct a set $G = \{G_{i,j} : i,j < k\}$, and exhibit a $\seq{P_0,\ldots,P_{k-1}} \in \mathcal{C}$, such that $G_{i,j} \subseteq P_j$ for all $i,j < k$, and $c$ avoids the color $i$ on $G_{i,j}$. We will furthermore satisfy the following requirement for each $n \in \omega$ and all $\alpha \in k^k$:
	\[
		\mathcal{R}_{n,\alpha} : (\exists j < k)(\exists x \geq n)[x \in G_{\alpha(j),j}].
	\]
	The claim is that $c$ is then balanced on some $P_j$. For if not, define $\alpha \in k^k$ by letting $\alpha(j)$ be the color $i < k$ such that every infinite homogeneous set for $c$ contained in $P_j$ has color $i$. Since $G$ satisfies $\mathcal{R}_{n,\alpha}$ for all $n$, there must be a $j < k$ such that $G_{\alpha(j),j}$ is infinite. Let $H$ be any infinite homogeneous set for $c$ contained in $G_{\alpha(j),j}$. As $c$ avoids the color $\alpha(j)$ on $G_{\alpha(j),j}$, it follows that $H$ has some other color than $\alpha(j)$, which is a contradiction since $G_{\alpha(j),j} \subseteq P_j$.
	
	The construction of $G$ is by a forcing notion whose conditions are tuples
	\[
		p = (\{E_{i,j} : i,j < k\}, X, \mathcal{D}),
	\]
	such that for all $i,j < k$:
	\begin{itemize}
		\item $E_{i,j}$ is a finite subset of $A$;
		\item $X$ is a computable infinite subset of $A$ such that $\max E_{i,j} < \min X$;
		\item for every $x \in X$, $c$ avoids the color $i$ on $E_{i,j} \cup \{x\}$;
		\item $\mathcal{D}$ is a nonempty $\Pi^0_1$ subclass of $\mathcal{C}$ such that for every $\seq{P_0,\ldots,P_{k-1}} \in \mathcal{D}$, $E_{i,j} \subseteq P_j$.
	\end{itemize}
	A condition $q = (\{F_{i,j} : i,j < k\}, Y, \mathcal{E})$ extends $p$ if $Y \subseteq X$, $\mathcal{E} \subseteq \mathcal{D}$, and $E_{i,j} \subseteq F_{i,j} \subseteq E_{i,j} \cup X$ for all $i,j < k$.
	
	Say a condition $p$ as above \emph{satisfies} $\mathcal{R}_{n, \alpha}$ if there are some $j < k$ and some $x \geq n$ such that $x \in E_{\alpha(j),j}$. We claim that the set of conditions satisfying $\mathcal{R}_{n,\alpha}$ is dense. Fix $p = (\{E_{i,j} : i,j < k\}, X, \mathcal{D})$. First, suppose there are some $\seq{Q_0,\ldots,Q_{k-1}} \in \mathcal{D}$, some $\ell < k$, and some $x \in X \cap Q_{\ell}$ such that $Y = \{y \in X : c(x,y) \neq \alpha(\ell)\}$ is infinite. Let $q = (\{F_{i,j} : i,j < k\}, Y, \mathcal{E})$, where $F_{\alpha(\ell),\ell} = E_{\alpha(\ell),\ell} \cup \{x\}$, $F_{i,j} = E_{i,j}$ for all $i,j < k$ with $i \neq \alpha(\ell)$ or $j \neq \ell$, and $\mathcal{E} = \{\seq{P_0,\ldots,P_{k-1}} \in \mathcal{D} : x \in P_\ell \}$. Then $q$ is an extension of $p$ satisfying $\mathcal{R}_{n,\alpha}$. So suppose now that there are no such $\seq{Q_0,\ldots,Q_{k-1}}$, $\ell$, and $x$. We derive a contradiction. The assumption implies that for every $x \in X$, $\lim_{y \in X} c(x,y)$ exists, since given any $\seq{Q_0,\ldots,Q_{k-1}} \in \mathcal{D}$, we have that $\lim_{y \in X} c(x,y) = \alpha(\ell)$ for the unique $\ell$ with $x \in Q_\ell$. So the map $g : X \to k$ defined by $g(x) = \lim_{y \in X} c(x,y)$ for all $x \in X$ is computable from every member of $\mathcal{D}$. By the cone-avoidance basis theorem (see, e.g.,~\cite[Theorem 2.1]{DJ-2009}), this implies that $g$ is computable. But then $c$ has a computable infinite homogeneous set, which we assumed it did not.
		
	To complete the proof, let $\mathcal{G} = \seq{q_0, q_1, \ldots}$ be a sufficiently generic sequence on our forcing poset, where for each $s$,
	\[
		q_s = (\{E^s_{i,j} : i,j < k\}, X^s, \mathcal{D}^s),
	\]
	and $q_s$ is extended by $q_{s+1}$. Define
	\[
		G_{i,j} = \bigcup_{s \in \omega} E_{i,j}^s
	\]
	for all $i,j < k$. Let $\seq{P_0,\ldots,P_{k-1}}$ be any element of $\bigcap_{s \in \omega} \mathcal{D}_s$, which is an intersection of a nested sequence of $\Pi^0_1$ classes and hence is nonempty. Then $G = \{G_{i,j} : i,j < k\}$ and $\seq{P_0,\ldots,P_{k-1}}$ have the desired properties.
\end{proof}

The following important problems arise frequently in the study of Weihrauch degrees.

\begin{definition}
\
	\begin{enumerate}
		\item $\LPO$ is the principle whose instances are all infinite binary sequences of the form $0^\omega$ or $0^n1^\omega$ for some $n \in \omega$, and the solutions are either the singleton $\{0\}$ if the instance is $0^\omega$, or $\{1\}$ if the instance is $0^n1^\omega$ for some $n$.
		\item $\NON$ is the principle whose instances are all sets, and the solutions to an instance $X$ are all sets $Y \nleq\sub{T} X$.
	\end{enumerate}
\end{definition}

Viewed as a $\Pi^1_2$ principle, $\NON$ is thus equivalent over $\RCA_0$ to the principle $\mathsf{AST}$ considered by Hirschfeldt, Shore, and Slaman~\cite[Section 6]{HSS-2009}. (See specifically~\cite[Theorem 6.3]{HSS-2009}.)

\begin{theorem}\label{T:NONLPORT2}
	$\LPO \times \NON \nured \RT^2_2$.
\end{theorem}

\begin{proof}
  Assume otherwise, and fix functionals $\Phi$ and $\Psi$ witnessing the reduction. We build an instance $S$ of $\LPO$ such that the pair $\seq{S,\emptyset}$ contradicts this assumption. We have that $\Phi^{S \oplus \emptyset}$ is a coloring $[\omega]^2 \to 2$, and for every infinite homogeneous set $H$ for this coloring, $\Psi^{S \oplus \emptyset \oplus H} = \seq{\{b\},Y}$, where $b$ is $0$ or $1$ depending on whether $S = 0^\omega$ or $S = 0^n1^\omega$ for some $n$, and $Y \nleq\sub{T} \emptyset$. (We think of $\Psi^{S \oplus \emptyset \oplus H}$ as the characteristic function of $\{b\} \oplus Y$, so that
$\Psi^{S \oplus \emptyset \oplus H}(0)\converges = 1$ if and only if $b = 0$.) We show that the coloring $\Phi^{S \oplus \emptyset}$ necessarily has an infinite homogeneous set $H$ satisfying one of the following properties:
	\begin{enumerate}
		\item\label{RT22NONCase1} $H$ is computable;
		\item\label{RT22NONCase2} $\Psi^{S \oplus \emptyset \oplus H}(0) \simeq 0$ and $S = 0^\omega$;
		\item\label{RT22NONCase3} $\Psi^{S \oplus \emptyset \oplus H}(0) \converges = 1$ and $S = 0^n1^\omega$ for some $n$.
	\end{enumerate}
	In the first case, $S \oplus \emptyset \oplus H$ obviously cannot compute a solution to our $\NON$-instance. And in the remaining cases, we have a contradiction to $\Psi^{S \oplus \emptyset \oplus H}$ giving us a solution to our $\LPO$-instance.
	
	Let $c$ be the coloring $\Phi^{0^\omega \oplus \emptyset} : [\omega]^2 \to 2$. Define $\mathcal{C}$ to be the $\Pi^0_1$ class consisting of all $2$-partitions $\seq{P_0,P_1}$ of $\omega$ such that
	\[
		(\forall i < 2)(\forall \text{ finite } F \subseteq P_i)[(\forall x,y \in F)[c(x,y) = i] \to \Psi^{0^\omega \oplus \emptyset \oplus F}(0) \simeq 0].
	\]
	We consider two cases. First, suppose $\mathcal{C}$ is nonempty. By Lemma \ref{lem:RT22NON} with $k = 2$ and $A = \omega$, if $c$ is unbalanced on $P_0$ and $P_1$ for every $\seq{P_0, P_1} \in \mathcal{C}$, then $c$ has a computable infinite homogeneous set. We can then take this to be $H$, set $S = 0^\omega$, and satisfy Property \eqref{RT22NONCase1} above. So assume not. Fix $\seq{P_0, P_1} \in \mathcal{C}$ and $i < 2$ such that $c$ is balanced on $P_i$, so that in particular, $P_i$ is infinite. Let $H \subseteq P_i$ be any infinite homogeneous set for $c$ with color $i$. If we then take $S = 0^\omega$, it follows by the definition of $\mathcal{C}$ that $\Psi^{S \oplus \emptyset \oplus H}(0) \simeq 0$, so we satisfy Property \eqref{RT22NONCase2}.
	
	So now, suppose $\mathcal{C} = \emptyset$. By compactness, choose $m$ so that for every partition $\seq{P_0, P_1}$ of $\omega$, there are an $i < 2$ and a finite $F \subseteq P_i \res m$ such that $c(x,y) = i$ for all $x,y \in F$ and $\Psi^{0^\omega \oplus \emptyset \oplus F}(0) \converges = 1$. Note that there are only finitely many such $F$ across all possible partitions, so there is a global bound $u$ on the uses of all these computations. Without loss of generality, $u \geq m$. Choose $n > u$ large enough so that $\Phi^{0^n1^\omega \oplus \emptyset}$ agrees with $c = \Phi^{0^\omega \oplus \emptyset}$ below $u$. Let $S = 0^n 1^\omega$, and let $d = \Phi^{S \oplus \emptyset}$. By repeatedly taking subsets, we see that there is a computable infinite set $Y$ such that $\min Y > m$ and for each $x < m$, $\lim_{y \in Y} d(x,y)$ exists. For each $i < 2$, let $Q_i = \{ x < m : \lim_{y \in Y} d(x,y) = i \}$, so that for some partition $\seq{P_0,P_1}$ of $\omega$, we have $Q_0 = P_0 \res m$ and $Q_1 = P_1 \res m$. Choose $i < 2$ and $F \subseteq Q_i$ as above. If $Y$ contains no infinite homogeneous set for $d$ with color $i$, then by~\cite[Theorem 5.11]{Jockusch-1972}, $Y$ contains a computable infinite homogeneous set with color $1-i$, and we satisfy Property \eqref{RT22NONCase1} again. Otherwise, we can take an infinite homogeneous set $H$ for $d$ having $F$ as an initial segment, and by construction, this set satisfies $\Psi^{S \oplus \emptyset \oplus H}(0) \converges = 1$ even though $S \neq 0^\omega$. Thus, we satisfy Property \eqref{RT22NONCase3}.
\end{proof}

Trivially, $\LPO \ured \RT^1_2$, since every instance of $\LPO$ can be regarded as an instance of $\RT^1_2$, so we have the following.

\begin{corollary}
	$\RT^1_2 \times \NON \nured \RT^2_2$.	
\end{corollary}

Clearly $\RT^1_2 \ured \SRT^2_2$, and there is a uniform construction of an $X$-computable instance of $\COH$ with no $X$-computable solution, so we also have the following.

\begin{corollary}\label{cor:Vasco}
	$\SRT^2_2 \times \COH \nured \RT^2_2$.
\end{corollary}

\begin{corollary}
	No additional relations hold between the problems in Theorem \ref{thm:RT22reds}.
\end{corollary}

\begin{proof}
	Since every computable instance of each of $\COH$ and $\SRT^2_2$ admits a $\Delta^0_2$ solution, so does every instance of $\COH \times \SRT^2_2$. By contrast, it is known that $\RT^2_2$ has a computable instance with no $\Delta^0_2$ solution. Thus, $\RT^2_2 \nured \COH \times \SRT^2_2$. The remaining non-reductions follow from this fact and Corollary \ref{cor:Vasco} by transitivity.
\end{proof}

Let $\limprob$ be the problem where an instance is a convergent sequence of elements of $\mathbb N^{\mathbb N}$, and the unique solution to this problem is the limit of this sequence. The following fact about $\limprob$ is well-known, but it is worth mentioning that it follows directly from Theorem \ref{T:NONLPORT2}, since $\LPO \times \NON \leq\sub{W} \limprob$.

\begin{corollary}
$\limprob \nured \RT^2_2$.
\end{corollary}

It is an interesting open question whether $\LPO$ can be replaced by an even weaker combinatorial principle in Theorem \ref{T:NONLPORT2}. A first candidate would be $\LLPO \uequiv \C_2$ (see~\cite[\S 3]{kreuzer}). Already $\LLPO \times \NON \nured \RT^2_2$ would imply $\WKL \nured \RT^2_2$, which is known by Liu's celebrated result~\cite{liu}. An even further improvement to $\ACC_n \times \NON \nured \RT^2_2$ would have as a consequence that $\mathsf{DNC}_n \nured \RT^2_2$ (see~\cite[\S 3, \S 5]{kreuzer}, which also has definitions of $\ACC_X$ and $\mathsf{DNC}_X$). On the other hand, it is the case that $\ACC_\mathbb{N} \times \NON \ured \mathsf{DNC}_\mathbb{N} \ured \mathsf{RT}^2_2$ (the latter reduction having been shown in~\cite[Theorem 2.3]{hirschfeldt3}).\footnote{The remarks in this paragraph were kindly provided by an anonymous referee.}

Note that Theorem \ref{T:NONLPORT2} also cannot be improved to show that $\LPO \times \NON \nured \RT^2_k$ for arbitrary $k \geq 2$. Indeed, each of $\LPO$ and $\NON$ is Weihrauch reducible to $\RT^2_2$, and so $\LPO \times \NON \ured \RT^2_2 \times \RT^2_2 \ured \RT^2_4$.

We can improve on this reduction with the following
strong counterpoint to Theorem \ref{T:NONLPORT2}, which shows that the theorem fails as soon as the number of colors is allowed to increase from two, even via a stable coloring. It is easy to see that $\LPO \ured \SRT^1_2$ and $\NON \ured \SRT^2_2$, so the following result also follows from~\cite[Theorem 3.24]{BR-2017}, which has as a special case that $\SRT^1_2 \times \SRT^2_2 \ured \SRT^2_3$.

\begin{proposition}\label{prop:LPONONSRT23}
	$\LPO \times \NON \ured \SRT^2_3$.
\end{proposition}

\begin{proof}
	Let $S$ be an arbitrary instance of $\LPO$, and let $X$ be any set. Let $c : [\omega]^2 \to 2$ be the result of applying a standard uniform construction of an $X$-computable stable coloring $c : [\omega]^2 \to 2$ with no $X$-computable infinite homogeneous set (e.g., as in~\cite[Theorem 2.1]{Jockusch-1972}). Define $d : [\omega]^2 \to 3$ by
	\[
	d(x,y) =
	\begin{cases}
		c(x,y) & \text{if } S(x) = S(y),\\
		2 & \text{otherwise.}
	\end{cases}
	\]
	Clearly, $d$ is uniformly computable from $S \oplus X$. If $S = 0^\omega$ then $d = c$, while if $S = 0^n1^\omega$ for some $n$ then $\lim_y d(x,y) = 2$ for all $x < n$, and $d(x,y) = c(x,y)$ for all $x \geq n$. Hence, every infinite homogeneous set for $d$ has color $0$ or $1$, and is also homogeneous for $c$. In particular, no infinite homogeneous set for $d$ is $X$-computable. Moreover, for any such infinite homogeneous set $H$, we have that $S = 0^n1^\omega$ if and only if $(\exists x \leq \min H)[S(x) = 1]$. Hence, $\seq{\{b\},H}$, where $b$ is $0$ or $1$ depending as $S = 0^\omega$ or $S = 0^n1^\omega$ for some $n$, is a uniformly $(S \oplus X \oplus H)$-computable solution to the $\LPO \times \NON$-instance $\seq{S,X}$.
\end{proof}

We do not know whether $\SRT^2_3$ above can be replaced by $\D^2_3$. However, we have the following related result, which does work for $\D^2_3$. The proof uses a novel coding mechanism.

\begin{theorem}
	For every $k \geq 1$, $\fe{(\RT^1_k)} \times \NON \ured \D^2_{k+1}$.	
\end{theorem}

\begin{proof}
	Let $c : \omega \to k$ be a coloring and $X$ a set. We describe a uniform procedure to define an $X$-computable stable coloring $d : [\omega]^2 \to k+1$ with no $X$-computable solution (i.e., no $X$-computable infinite limit-homogeneous set), and a uniform procedure for turning any such solution into an almost limit-homogeneous set for $c$. Fix a canonical $(c \oplus X)$-computable enumeration of $(c \oplus X)'$, and let
	\[
		s_0 < s_1 < \cdots
	\]
	be a $(c \oplus X)'$-computable sequence such that for all $e$, we have that $(c \oplus X)'[s_e]$ and $(c \oplus X)'$ agree on all $x \leq e$. Using $(c \oplus X)'$, choose
	\[
		x_{0,0} < x_{0,1} < x_{1,0} < x_{1,1} < x_{2,0} < x_{2,1} < \cdots
	\]
	with $x_{e+1,0} - x_{e,1} \geq s_e$ for all $e$, and such that either $\Phi^X_e(x_{e,0}) \converges = \Phi^X_e(x_{e,1}) \converges = 1$, or $\Phi^X_e(x) \simeq 0$ for all sufficiently large $x$.
	
	Now define a $(c \oplus X)'$-computable $(k+1)$-partition $P_0 \cup \cdots \cup P_k$ of $\omega$ as follows. For each $e$, put $x_{e,0}$ into $P_k$, and put every other $x$ into $P_{c(x)}$. Thus, for all $e$, we have that $x_{e,0}$ and $x_{e,1}$ belong to different parts of the partition, so by construction, if $\Phi^X_e$ defines an infinite set, this set cannot be an infinite subset of any $P_i$. We also have that if $z_0 < \cdots < z_{n-1} \in P_k$ then $z_{e+1} - z_e \geq s_e$ for all $e < n$, so any infinite subset of $P_k$ computes $(c \oplus X)'$. We can regard  $P_0 \cup \cdots \cup P_k$ as a $(c \oplus X)$-computable stable coloring $d$. Clearly, $d$ is defined uniformly from $X$ and $c$, and
no infinite limit-homogeneous set for $d$ is $X$-computable.

	Consider any $\mathsf{D}^2_{k+1}$-solution to $d$, i.e., any infinite set $Z = \{z_0 < z_1 < \cdots\}$ contained in one of $P_0,\ldots,P_k$. We construct a set $Y = \{y_0 < y_1 < \cdots\}$ inductively by stages, defining $y_n$ at stage $n$. At any stage, we may choose to \emph{exit the construction}, which simply means to let $m$ be the maximum of all $y_n$ defined thus far, and let the rest of our set be $\{z_n \in Z : z_n > m\}$. At stage $n = 0$, let $y_0 = 0$, and declare no color $i < k$ \emph{forbidden}. If we have not exited the construction by stage $n+1$, assume we have defined $y_n$ and there is at least one $i < k$ that is still not forbidden. For the least such $i$, we compute a number $e$ such that $e \in (c \oplus X)'$ if and only if $(\exists y > y_n )[c(y) = i]$, which we can do uniformly from $y_n$, the color $i$, and an index for $c$ as a $(c \oplus X)$-computable coloring. If $e \in (c \oplus X)'[z_{e+1} - z_e]$ then certainly $e \in (c \oplus X)'$, so we can find a $y > y_n$ with $c(y) = i$, and we let $y_{n+1}$ be the least such $y$. If $e \notin (c \oplus X)'[z_{e+1} - z_e]$, we declare $i$ forbidden and restart the process with the next smallest non-forbidden color. In this case, we promise that if at any future stage we see a $y > y_n$ with $c(y) = i$, we exit the construction. Note that this can happen only if $z_{e+1} - z_e < s_e$. Note also that it must happen if all $i<k$ become forbidden.
	
	It is easy to see that $Y$ is uniformly computable from $c \oplus X \oplus Z$. We claim that $Y$ is almost limit-homogeneous for $c$. This is clear if we never exit the construction, because in that case there must be some least $i$ that is never declared forbidden, and then $c(y_n) = i$ for almost all $n$. If, on the other hand, we do exit the construction, then as noted above we must have $z_{e+1} - z_e < s_e$ for some $e$, and hence $Z$ cannot be a subset of $P_k$. In this case, $Z$ is therefore a subset of $P_i$ for some $i < k$, and by construction, if $x \in P_i$ for such an $i$ then $c(x) = i$. As $Y =^* Z$, it follows that $Y$ is almost limit-homogeneous for $c$.
\end{proof}

\section{Stable Ramsey's theorem for pairs}


As mentioned above, every instance of $\LPO$ can be regarded as an instance of $\RT^1_2$. The latter instance, however, is consequently unbalanced. It is interesting to ask whether this is the only possible reduction, or whether $\LPO$ can in fact be reduced to $\RT^1_2$ via a balanced coloring. The following proposition shows that the answer is no. It also points to an additional point of disagreement in the uniform strengths of $\SRT^2_2$ and $\D^2_2$, to complement the aforementioned result that $\SRT^2_2 \nured \D^2_k$ for all $k$.

We first give a definition.

\begin{definition}
	For $\mathsf{P} \in \{\RT^1_k,\SRT^2_k,\D^2_k\}$, let $\bal{\mathsf{P}}$ be the restriction of $\mathsf{P}$ to balanced colorings (on $\omega$), and $\unbal{\mathsf{P}}$ the restriction to unbalanced colorings.
\end{definition}

\begin{proposition}\label{Prop:LPO_SRT_D}
	$\LPO \ured \bal{\SRT^2_2}$, but $\LPO \nured \bal{\D^2_k}$ for all $k$.	
\end{proposition}

\begin{proof}
	For the positive reduction, let $S$ be any instance of $\LPO$. Let $\mathrm{par}(x)$ be $0$ or $1$ depending on whether $x$ is even or odd, and define $c : [\omega]^2 \to 2$ by
	\[
		c(x,y) =
		\begin{cases}
			\mathrm{par}(x) & \text{if } (\forall z < y)[S(z) = 0],\\
			1 - \mathrm{par}(x) & \text{otherwise}.
		\end{cases}
	\]
	Thus, if $S = 0^\omega$ then $\lim_y c(x,y) = \mathrm{par}(x)$ for all $x$, and if $S = 0^n1^\omega$ for some $n$ then $\lim_y c(x,y) = 1 - \mathrm{par}(x)$ for all $x$. In either case, for each $i < 2$, there are infinitely many $x$ with $\lim_y c(x,y) = i$, so $c$ is balanced. Now, every element in an infinite homogeneous set for $c$ has the same parity. So if $H$ is any such homogeneous set, and if $x_0$ and $x_1$ are its least two elements, then $S = 0^\omega$ if and only if $c(x_0,x_1) = \mathrm{par}(x_0)$. Thus, we have the desired uniform reduction.
	
	For the negative reduction, assume towards a contradiction that $\LPO \ured \bal{\D^2_k}$ via some $\Phi$ and $\Psi$. Then $c = \Phi^{0^\omega}$ is a computable balanced stable coloring $[\omega]^2 \to k$. Define $\mathcal{C}$ to be the $\Pi^0_1$ class of all partitions $\seq{P_0,\ldots,P_{k-1}}$ of $\omega$ such that
	\[
		(\forall i < k)(\forall \text{ finite } F \subseteq P_i)[\Psi^{0^\omega \oplus F}(0) \simeq 0].
	\]
	First, we claim that $\mathcal{C} \neq \emptyset$. Otherwise, by compactness, there is an $m$ such that for every partition $\seq{P_0,\ldots,P_{k-1}}$ of $\omega$, there are an $i < k$ and a finite $F \subseteq P_i \res m$ such that $\Psi^{0^\omega \oplus F}(0) \converges = 1$. Let $u \geq m$ be a bound on the uses of all these computations, for all possible such $F$. Choose $n > u$ large enough so that $\Phi^{0^n1^\omega}$ and $c = \Phi^{0^\omega}$ agree below $u$. Let $S = 0^n1^\omega$ and $d = \Phi^S$, which is another balanced stable coloring. For the partition $\seq{P_0,\ldots,P_{k-1}}$ of $\omega$ given by $P_i = \{x : \lim_y d(x,y) = i\}$, fix $i < k$ and $F \subseteq P_i \res m$ as above. As $d$ is balanced, there is an infinite homogeneous set $H$ for $d$ with color $i$ that has $F$ as an initial segment. But then we have $\Psi^{S \oplus H}(0) \converges = 1$ even though $S \neq 0^\omega$, a contradiction. So $\mathcal{C} \neq \emptyset$. Choose any $\seq{P_0,\ldots,P_{k-1}} \in \mathcal{C}$. Since $P_0 \cup \cdots \cup P_{k-1} = \omega$, there is an $i < k$ such that $P_i$ is infinite. Let $L$ be any infinite limit-homogeneous set for $c$ contained in $P_i$. Then $\Psi^{0^\omega \oplus L}(0) \simeq 0$, which contradicts the choice of $\Psi$.
\end{proof}

\begin{corollary}
	$\LPO \nured \bal{\RT^1_k}$ for all $k$.	
\end{corollary}

\begin{proof}
	The usual proof that $\RT^1_k \ured \D^2_k$ shows that $\bal{\RT^1_k} \ured \bal{\D^2_k}$.	
\end{proof}

One generalization of the notion of unbalanced coloring is the following, in which merely one of the possible colors of homogeneous sets---rather than, all but one---is omitted.

\begin{definition}
	Let $c : [\omega]^2 \to k$ be a coloring and $X$ a set. The coloring $c$ is \emph{thin-unbalanced on $X$} if for some $i < k$, there is no infinite homogeneous set for $c$ contained in $X$ with color $i$. The color $i$ is called \emph{a witness of thin-unbalancing for $c$ on $X$}. If $c$ is not thin-unbalanced on $X$, it is \emph{thin-balanced on $X$}.
\end{definition}

When $X = \omega$, we shall simply say $c$ is thin-unbalanced / thin-balanced. Note that if $k = 2$, then $c$ is thin-unbalanced on a set if and only if it is unbalanced on that set in the sense of Definition \ref{def:unbalanced}, which in turn holds if and only if $c$ avoids one of its two colors on that set.

\begin{definition}
For $\mathsf{P} \in \{\SRT^2_k, \mathsf{D}^2_k\}$, we define the following variations on $\mathsf{P}$:
\begin{itemize}
	\item $\dwub{\mathsf{P}}$ is the problem whose instances are pairs $\seq{c,\ell}$ where $c$ is a thin-unbalanced instance of $\mathsf{P}$ and $\ell : \omega \to k$ is a function such that $\lim_y \ell(y)$ exists and is a witness of thin-unbalancing for $c$, and the solutions to such a pair are the $\mathsf{P}$-solutions to $c$.
	\item $\wub{\mathsf{P}}$ is the problem whose instances are pairs $\seq{c,i}$ where $c$ is a thin-unbalanced instance of $\mathsf{P}$ and $i$ is a witness of thin-unbalancing for $c$, and the solutions to such a pair are the $\mathsf{P}$-solutions to $c$.
\end{itemize}
\end{definition}

The above are arguably not natural problems from a combinatorial point of view, and we will not study them in their own right. Rather, our interest is in what they can reveal about $\SRT^2_k$ and $\D^2_k$. As we will see, the above restrictions capture various elements of standard proofs of the latter principles.

\begin{proposition}
	\
	\begin{enumerate}
		\item For $\mathsf{P} \in \{\SRT^2_k, \D^2_k\}$, we have $\LPO \ured \wub{\mathsf{P}} \sured \dwub{\mathsf{P}}$. 
		\item $\wub{\SRT^2_k} \times \wub{\SRT^2_2} \sured \wub{\SRT^2_k}$ and $\wub{\D^2_k} \times \wub{\D^2_2} \sured \wub{\D^2_k}$.
	\end{enumerate}
\end{proposition}

\begin{proof}
	For part (1), it is enough to show that $\LPO \ured \wub{\D^2_2}$, the rest of the reductions being obvious. This is proved much like Proposition \ref{prop:LPONONSRT23}. Given an instance $S$ of $\LPO$, define $c : [\omega]^2 \to 2$ by
	\[
	c(x,y) =
	\begin{cases}
		1 & \text{if } S(x) = S(y),\\
		0 & \text{otherwise.}
	\end{cases}
	\]
	If $S = 0^\omega$ then $c(x,y) = 1$ for all $x < y$, and if $S = 0^n1^\omega$ then $\lim_y c(x,y) = 1$ for all $x > n$ and $\lim_y c(x,y) = 0$ for all $x \leq n$. Either way, $c$ is an instance of $\wub{\D^2_2}$ with witness of thin-unbalancing $0$. Now if $L$ is any limit-homogeneous set for $c$ then $S = 0^\omega$ if and only if there is an $x \leq \min L$ such that $S(x) = 1$.
	
	For part (2), we prove the result for $\SRT^2_k$, the proof for $\D^2_k$ being similar. Let $c : [\omega]^2 \to k$ and $d : [\omega]^2 \to 2$ be instances of $\wub{\SRT^2_k}$ and $\wub{\SRT^2_2}$, respectively. Say the witnesses of thin-unbalancing for $c$ and $d$ are $i_c < k$ and $i_d < 2$, respectively. Define $e : [\omega]^2 \to k$ by
	\[
	e(x,y) =
	\begin{cases}
		i_c & \text{if } c(x,y) = i_c \text{ or } d(x,y) = i_d,\\
		c(x,y) & \text{otherwise.}
	\end{cases}
	\]
        Notice that $e$ is stable.
	We claim that $e$ is thin-unbalanced as witnessed by $i_c$. Indeed, if $H$ were infinite and homogeneous for $e$ with color $i_c$ then we could define $f : [H]^2 \to 2$ by
	\[
		f(x,y) =
	\begin{cases}
		0 & \text{if } c(x,y) = i_c,\\
		1 & \text{otherwise.}
	\end{cases}
	\]
	Any infinite homogeneous set for $f$ contained in $H$ with color $0$ would be homogeneous for $c$ with color $i_c$, and any infinite homogeneous set for $f$ contained in $H$ with color $1$ would be homogeneous for $d$ with color $i_d$. Neither of these is possible by assumption, so the claim holds. Hence, $e$ is an instance of $\wub{\SRT^2_k}$, and it is clear that any infinite homogeneous set for $e$ is homogeneous for both $c$ and $d$.
\end{proof}

As we will see in Proposition \ref{CFIprop}, $\wub{\mathsf{P}} \uequiv \dwub{\mathsf{P}}$ for $\mathsf{P} \in \{\SRT^2_2, \D^2_2\}$.
Note that in part (1) above, the reduction from $\LPO$ to $\wub{\mathsf{P}}$ cannot be improved from $\ured$ to $\sured$. Indeed, it follows from a result of Brattka and Rakotoniaina~\cite[Corollary 3.15]{BR-2017} that $\mathsf{LPO} \nsured \RT^n_k$ for all $n,k \geq 1$.


\begin{theorem}\label{T:Dchar}
	Let $\mathsf{P}$ be a problem. Then $\mathsf{P} \ured \dwub{\D^2_k}$ if and only if $\LPO \times \mathsf{P} \ured \D^2_k$.
\end{theorem}

\begin{proof}
	For the forward direction, we prove that $\LPO \times \dwub{\D^2_k} \ured \D^2_k$. Again, we emulate the proof of Proposition \ref{prop:LPONONSRT23}. Let $S \in 2^\omega$ be an instance of $\LPO$. Let $\seq{c,\ell}$ be an instance of $\dwub{\D^2_k}$, so that $c$ is a stable coloring $[\omega]^2 \to k$, and $\ell : \omega \to k$ is a function with $\lim_y \ell(y) = i < k$ a witness of thin-unbalancing for $c$. Define $d : [\omega]^2 \to k$ by
	\[
	d(x,y) =
	\begin{cases}
		c(x,y) & \text{if } S(x) = S(y),\\
		\ell(y) & \text{otherwise.}
	\end{cases}
	\]
	Thus, if $S = 0^\omega$ then $d = c$, and if $S = 0^n1^\omega$ for some $n$ then $c(x,y) = d(x,y)$ for all $x \geq n$ and $\lim_y d(x,y) = \lim_y \ell(y) = i$ for all $x < n$. Since $c$ has no infinite limit-homogeneous set with color $i$, it follows that every infinite limit-homogeneous set $L$ for $d$ is also limit-homogeneous for $c$. Moreover, we have that $S = 0^n1^\omega$ if and only if $(\exists x \leq \min L)[S(x) = 1]$. Hence, $\seq{\{b\},L}$, where $b$ is $0$ or $1$ depending on whether $S = 0^\omega$ or $S = 0^n1^\omega$ for some $n < L$, is a uniformly $S \oplus L$-computable solution to the $\LPO \times \dwub{\D^2_k}$-instance $\seq{S,\seq{c,\ell}}$.
	
	For the reverse direction, fix a problem $\mathsf{P}$ such that $\LPO \times \mathsf{P} \ured \D^2_k$, say via functionals $\Phi$ and $\Psi$. Fix an instance $X$ of $\mathsf{P}$. We describe a uniform procedure to define an $X$-computable thin-unbalanced stable coloring $d : [\omega]^2 \to k$ with a witness given in a $\Delta^0_2$ way, and a uniform procedure for turning any infinite limit-homogeneous set for $d$ into a $\mathsf{P}$-solution for $X$. To begin, let $c = \Phi^{0^\omega \oplus X}$, which is a stable coloring $[\omega]^2 \to k$ by assumption. Define $\mathcal{C}$ to be the $\Pi^{0,X}_1$ class consisting of all partitions $\seq{P_0, \ldots, P_{k-1}}$ of $\omega$ such that
	\[
		(\forall i < k)(\forall \text{ finite } F \subseteq P_i)[\Psi^{0^\omega \oplus X \oplus F}(0) \simeq 0].
	\]
(As in the proof of Theorem~\ref{T:NONLPORT2}, we think of $\Psi^{0^\omega \oplus X \oplus F}$ as the characteristic function of a join.)

        It must be that $\mathcal{C} = \emptyset$. For suppose otherwise, and choose any $\seq{P_0, \ldots, P_{k-1}} \in \mathcal{C}$ and an $i < k$ such that $P_i$ is infinite. Let $L \subseteq P_i$ be an infinite limit-homogeneous set for $c$. Then by the definition of $\mathcal{C}$, we have $\Psi^{0^\omega \oplus X \oplus L}(0) \simeq 0$, which is a contradiction because $\seq{0^\omega,X}$ is an instance of $\LPO \times \mathsf{P}$, and we should thus have $\Psi^{0^\omega \oplus X \oplus L}(0) \converges = 1$. So $\mathcal{C}$ is empty, as claimed. By compactness, we can uniformly $X$-computably find an $m$ such that for every partition $\seq{P_0, \ldots, P_{k-1}}$ of $\omega$, there are an $i < k$ and a finite $F \subseteq P_i \res m$ such that $\Psi^{0^\omega \oplus X \oplus F}(0) \converges = 1$. Let $u \geq m$ be a bound on the uses of all these computations, for all possible such $F$. Choose $n > u$ large enough so that $\Phi^{0^n1^\omega \oplus X}$ and $c = \Phi^{0^\omega \oplus X}$ agree below $u$. Let $S = 0^n 1^\omega$, and let $d = \Phi^{S \oplus X}$. Note that $d$ is uniformly $X$-computable.
	
	We claim that $d$ is thin-unbalanced. To see this, let $\seq{P_0,\ldots,P_{k-1}}$ be the partition of $\omega$ given by $P_i = \{x \in \omega: \lim_y d(x,y) = i\}$. Let $i < k$ and $F \subseteq P_i \res m$ be as above. If $P_i$ were infinite, then there would be an infinite limit-homogeneous set $L$ for $d$ having $F$ as an initial segment, and by construction, this set would satisfy $\Psi^{S \oplus X \oplus L}(0) \converges = 1$ even though $S \neq 0^\omega$. Thus, $P_i$ is finite, so $i$ is a witness to thin-unbalancing for $d$. Moreover, since $i$ depends only on $\lim_y d(x,y)$ for $x < m$, it follows that $i$ can be approximated from $d$, and hence from $X$, in a uniform $\Delta^0_2$ way. So $d$ is an instance of $\dwub{\D^2_k}$. Now if $L$ is any infinite limit-homogeneous set for $d$, we must have $\Psi^{S \oplus X \oplus L} = \seq{\{1\}, Y}$, where $Y$ is a $\mathsf{P}$-solution to $X$. Hence, there is a uniform way to convert $X \oplus L$ into a $\mathsf{P}$-solution for $X$, as desired.
\end{proof}

A succinct way to express the characterization given by the preceding theorem is that $\dwub{\D^2_k} \equiv\sub{W} \sup_{\ured} \{ \mathsf{P} : \mathsf{P} \times \LPO \ured \D^2_k\}$. We can obtain several other results of this sort, the proofs of which are similar to the preceding theorem.

\begin{proposition}
	The following all exist and are all Weihrauch equivalent:
\begin{enumerate}
	\item $\dwub{\D^2_k}$;
	\item $\sup_{\ured} \{ \mathsf{P} : \mathsf{P} \times \LPO \ured \D^2_k\}$;
	\item $\sup_{\ured} \{ \mathsf{P} : \mathsf{P} \times \LPO \ured \dwub{\D^2_k}\}$;
	\item $\sup_{\ured} \{ \mathsf{P} : \mathsf{P} \times \dwub{\D^2_k} \ured \dwub{\D^2_k}\}$;
	\item $\sup_{\ured} \{ \mathsf{P} : \mathsf{P} \times \dwub{\D^2_k} \ured \D^2_k\}$.
\end{enumerate}
\end{proposition}

We do not know a similar characterization for $\SRT^2_k$, nor even an answer to the following question. (It is worth noting that the Weihrauch lattice is not complete. Indeed, by~\cite[Proposition 3.15]{higuchipauly}, it does not have any nontrivial infinite suprema.)

\begin{question}
	Does $\sup_{\ured} \{ \mathsf{P} : \mathsf{P} \times \LPO \ured \SRT^2_k\}$ exist?
\end{question}

\noindent As a partial step, we have the following:

\begin{proposition}
	Let $\mathsf{P}$ be a problem.
	\begin{enumerate}
		\item If $P \ured \dwub{\SRT^2_k}$ then $\LPO \times \mathsf{P} \ured \mathsf{SRT}^2_k$.
		\item If $\LPO \times \mathsf{P} \ured \mathsf{SRT}^2_k$
                  then $\mathsf{P}$ is Weihrauch reducible to the
                  problem whose instances are pairs $\seq{c,\ell}$
                  where $c$ is an instance of $\SRT^2_k$ and $\ell :
                  \omega \to k$ is a function such that $\lim_y
                  \ell(y)$ exists and is a witness to thin-unbalancing
                  for $c$ on some set that is low relative to $c$, and the solutions to such a pair are the $\SRT^2_k$-solutions to $c$.
	\end{enumerate}
\end{proposition}

\begin{proof}
	Part (1) is proved just like the forward direction of Theorem
        \ref{T:Dchar}. For part (2), we proceed as in the proof of the reverse direction of Theorem \ref{T:Dchar}, only the $\Pi^{0,X}_1$ class $\mathcal{C}$ now consists of all partitions $\seq{P_0,\ldots,P_{k-1}}$ of $\omega$ such that
	\[
		(\forall i < k)(\forall \text{ finite } F \subseteq P_i)[(\forall x,y \in F)[c(x,y) = i] \to \Psi^{0^\omega \oplus X \oplus F}(0) \simeq 0].
	\]
We can assume that $c \equiv\sub{T} X$, because we can replace it by
the coloring $c'$ obtained by letting $c'(n,n+1)=X(n)$ and
$c'(x,y)=c(x,y)$ for all other pairs. An infinite solution to
$c'$ can be uniformly transformed into one to $c$ by thinning.
	
	Now, if $\mathcal{C} = \emptyset$, let $n$ be as in the corresponding case in the proof of Theorem \ref{T:Dchar}. For each $i < k$, let $P_i = \{x : \lim_y \Phi^{0^n1^\omega \oplus X}(x,y) = i\}$. Then for some $i < k$, there exists a finite $F \subseteq P_i$ such that $F$ is homogeneous for $\Phi^{0^n1^\omega \oplus X}$ with color $i$ and $\Psi^{0^n1^\omega \oplus X \oplus F}(0) \converges = 1$. Moreover, this $F$ and $i$ can be approximated in a $\Delta^{0,X}_2$ way (i.e., $\Delta^0_2$ relative to the instance $\seq{0^n1^\omega,X}$). Now if $\Phi^{0^n1^\omega \oplus X}$ had any infinite homogeneous set with color $i$, then $c$ would have such a set extending $F$, which would produce the same contradiction as in Theorem \ref{T:Dchar}.  Thus, it must be that $\Phi^{0^n1^\omega \oplus X}$ has no homogeneous set with color $i$, so in particular, it is thin-unbalanced (on the low set $\omega$).
	
	If, on the other hand, $\mathcal{C} \neq \emptyset$, then let $\seq{P_0,\ldots,P_{k-1}}$ be the canonical low-over-$X$ element of it (given by the proof of the low basis theorem). Clearly, we can approximate in a $\Delta^{0,X}_2$ way (in fact, in a $\Pi^{0,X}_1$ way), the least $i$ such that $P_i$ is infinite. Then $c = \Phi^{0^\omega \oplus X}$ must be thin-unbalanced on $P_i$ with witness $i$. Otherwise, we could take a homogeneous set $H$ for $c$ with color $i$ contained in $P_i$ and have, by the definition of $\mathcal{C}$, that $\Psi^{0^\omega \oplus X \oplus H}(0) \simeq 0$, which is a contradiction because $\seq{0^\omega,X}$ is an instance of $\mathsf{LPO} \times \mathsf{P}$, and we should thus have $\Psi^{0^\omega \oplus X \oplus H}(0) \converges = 1$.
\end{proof}

\begin{remark}
	With a view to some of the recent work on the algebraic structure of the Weihrauch degrees (\cite{BP-TA,Dzhafarov-TA}), Theorem \ref{T:Dchar} suggests a natural \emph{parallel quotient} operator on problems, given by $\mathsf{P}/\mathsf{Q} = \sup_{\ured}\{\mathsf{R} : \mathsf{R} \times \mathsf{Q} \ured \mathsf{P}\}$. We have no reason to think this operator is total, but studying the kinds of problems for which it is defined ought to be interesting in its own right.
\end{remark}

\section{The cofinite-to-infinite principle}\label{s:CFI}

In this section, we briefly depart from studying products, to investigate $\wub{\D^2_2}$ (in the
guise of a Weihrauch-equivalent principle introduced below) in the context of other weak Weihrauch degrees. Some of our terminology will be specific to the Weihrauch literature,
and we refer the reader to~\cite{BGP-TA} for any definitions we omit.

We begin by introducing the following ``cofinite set to infinite set'' principle.

\begin{definition}\label{def:cfi}
	$\CFI_{\Delta^0_2}$ is the restriction of $\D^2_2$ to colorings $c : [\omega]^2 \to 2$ such that $\lim_y c(x,y) = 1$ for almost all $x$.
\end{definition}

\noindent (Thus, informally, $\CFI_{\Delta^0_2}$ is the problem of finding an infinite subset of a cofinite set given by a $\Delta^0_2$ approximation.) The connection to the previous section is provided by the following result.



\begin{proposition}
\label{CFIprop}
$\CFI_{\Delta^0_2} \uequiv \wub{\D^2_2} \uequiv \dwub{\D^2_2}$.
\end{proposition}

\begin{proof}
	It is clear that $\CFI_{\Delta^0_2} \ured \wub{\D^2_2} \ured \dwub{\D^2_2}$. In the other direction, suppose we are given an instance $\seq{c,\ell}$ of $\dwub{\D^2_2}$. Define $d : [\omega]^2 \to 2$ by
	\[
		d(x,y) =
		\begin{cases}
			1 & \text{if } c(x,y) = 1 - \ell(y),\\
			0 & \text{otherwise.}
		\end{cases}
	\]
	Now for all $x$, we have that $\lim_y c(x,y) = 1 - \lim_y \ell(y)$ if and only if $\lim_y d(x,y) = 1$. In particular, since $\lim_y c(x,y) = 1 - \lim_y \ell(y)$ for almost all $x$, we have $\lim_y d(x,y) = 1$ for almost all $x$. Clearly, every limit-homogeneous set for $d$ is also limit-homogeneous for $c$.
\end{proof}

\noindent Notice that a similar proof shows that $\wub{\SRT^2_2} \uequiv \dwub{\SRT^2_2}$.

We now compare $\CFI_{\Delta^0_2}$ with the choice principle $\C_{\mathbb{N}}$ defined in Section \ref{s:intro}.

\begin{proposition}
	$\CFI_{\Delta^0_2} \comp \mathsf{C}_{\mathbb{N}} \uequiv \CFI_{\Delta^0_2}$.
\end{proposition}

\begin{proof}
   First we show that $\CFI_{\Delta^0_2} \star \C_{\mathbb{N}} \ured \CFI_{\Delta^0_2} \times \C_{\mathbb{N}}$. Note that $\mathsf{C}_{\mathbb{N}}$ is computable with finitely many mindchanges, and these mindchanges can be incorporated into the $\Delta^0_2$ instances of $\CFI_{\Delta^0_2}$. Thus, we can compute directly the impact $\C_{\mathbb{N}}$ has on the $\CFI_{\Delta^0_2}$-instance, and do not need to use them sequentially.

   Then we argue that $ \CFI_{\Delta^0_2} \times \C_{\mathbb{N}} \ured \CFI_{\Delta^0_2} $. We identify an instance $e$ of $\C_{\mathbb{N}}$ with the complement of the set enumerated by $e$, and an instance $c$ of $\CFI_{\Delta^0_2}$ with the corresponding $\Delta^{0,c}_2$ set.
   As shown in~\cite[Lemma 2.3]{paulyfouchedavie}, we may assume without loss of generality that the instances of $\C_{\mathbb{N}}$ are of the form $\{n \mid n > k\}$ for some $k \in \N$. Given instances of $\C_{\mathbb{N}}$ and of $\CFI_{\Delta^0_2}$, we can compute the intersection of these instances, and think of it as an instance of  $\CFI_{\Delta^0_2}$. Any infinite subset of this intersection is a solution to the original $\CFI_{\Delta^0_2}$ instance, and any element a solution to the $\C_{\mathbb{N}}$ instance.
\end{proof}

We can think of the instances of $\CFI_{\Delta^0_2}$ as being functions $p : \omega \to \omega$ such that $|\{i \in \naturalnumbers : p(i) = n + 1\}| < \infty$ for all $n$, and such that $|\{i \in \naturalnumbers : p(i) = n + 1\}|$ is even for cofinitely many $n$. Then, a solution is any infinite set $Y$ such that if $n \in Y$ then $|\{i \in \naturalnumbers : p(i) = n + 1\}|$ is even. It is easy to see that this formulation is Weihrauch equivalent to the one given in Definition \ref{def:cfi}. However, we shall find this version more convenient for our results below.

Given $p \in \omega^\omega$ as above, let $\psi(p) = \{ n : |\{i \in \naturalnumbers : p(i) = n + 1\}| \text{ is even}\}$. For each $p \in {\omega}^{<\omega} \cup \Baire$, let $[p] = \{n : \exists i \ p(i) = n+1\}$. For $\sigma \in \omega^{<\omega}$, let $\oldoverline{\sigma}$ be the length-lexicographically least extension of $\sigma$ such that$|\{i \in \naturalnumbers : \oldoverline{\sigma}(i) = n + 1\}|$ is even for all $n \in [\sigma]$.

\begin{definition}
	For $\sigma \in \omega^{<\omega}$, let $\CFI_{\Delta^0_2}^\sigma$ denote the restriction of $\CFI_{\Delta^0_2}$ to instances $\oldoverline{\sigma}p$ such that $[\sigma] \cap [p] = \emptyset$.
\end{definition}

\begin{proposition}\label{prop:restr}
	For every $\sigma \in \omega^{<\omega}$, we have $\CFI_{\Delta^0_2}^\sigma \uequiv \CFI_{\Delta^0_2}$.
\end{proposition}

\begin{proof}
	Let $h : \naturalnumbers \to
	\{x \in \omega : x > \max [\sigma]\}$ be a computable bijection.
	Let $\overline{h} : \Baire \to \Baire$ be defined pointwise via $\overline{h}(p)(n) = h(p(n))$. Then $\CFI_{\Delta^0_2}(p) = h^{-1} \circ \CFI_{\Delta^0_2}^\sigma(\oldoverline{\sigma}\overline{h}(p))$.
\end{proof}

We can now prove that $\CFI_{\Delta^0_2}$ has properties very similar to being a \emph{total fractal} (see Brattka, Gherardi, and Pauly~\cite[Section 4]{BGP-TA}; see also Theorem 7.15 in that paper). In the context of Weihrauch degrees, a fractal may be thought of as a problem that retains its full power on arbitrarily small (clopen) restrictions of its domain. Since $\CFI_{\Delta^0_2}$ is defined on $\Delta^0_2$-approximations, it is clear that it is a fractal.

\begin{proposition}\label{prop:cbabsorption}
	Let $\mathsf{P}$ be any problem. If $\CFI_{\Delta^0_2} \ured \mathsf{P} \comp \C_{\mathbb{N}}$, then $\CFI_{\Delta^0_2} \ured \mathsf{P}$.
\end{proposition}

\begin{proof}
	Let $\Phi$ map instances of $\CFI_{\Delta^0_2}$ to instances of $\mathsf{C}_{\mathbb{N}}$.
	If there is no string $\sigma \in \omega^{<\omega}$ such that $\Phi^{\oldoverline{\sigma}}(0,s)\converges = 1$ for some $s$, then $0$ is always a valid answer to the $\mathsf{C}_{\mathbb{N}}$-instances used in the reduction, and the $\mathsf{C}_{\mathbb{N}}$-call is useless. So suppose otherwise, and let $\sigma_0$ be such a string.
	
	Assume now that we have defined $\sigma_0,\ldots,\sigma_{k-1} \in \omega^{<\omega}$, and choose the least $n \notin \bigcup_{i < k} [\sigma_i]$. Suppose there is no string $\sigma$ with $n \in [\sigma]$ and $[\sigma] \cap [\sigma_i] = \emptyset$ for all $i < k$, and such that $\Phi^{\oldoverline{\sigma}_0 \ldots \oldoverline{\sigma}_{k-1} \oldoverline{\sigma}}(k,s) \converges = 1$ for some $s$. Then choose any $\sigma$ with $n \in [\sigma]$ and $[\sigma] \cap [\sigma_i] = \emptyset$ for all $i < k$. Now $\CFI_{\Delta^0_2}^{\sigma_0 \cdots \sigma_{k-1}\sigma} \ured \mathsf{P}$ because we can replace the output of $\mathsf{C}_{\mathbb{N}}$ by $k$. By Proposition \ref{prop:restr}, this fact implies the claim. So suppose otherwise, and let $\sigma_k$ be a string with the desired properties.
	
	If this procedure never stops, then we construct some $p = \oldoverline{\sigma}_0\oldoverline{\sigma}_1\oldoverline{\sigma}_2\ldots$. By induction, all the $[\sigma_i]$ are mutually disjoint, and every $n \in \naturalnumbers$ appears some even number of times in some $[\sigma_i]$, so certainly $p$ is an instance of $\CFI_{\Delta^0_2}$. However, by construction we also find that for each $k$ there is an $s$ such that $\Phi^p(k,s)=1$, so $\Phi^p$ is not an instance of $\mathsf{C}_{\mathbb{N}}$, which is a contradiction.
\end{proof}

This proposition allows us to deduce a number of non-reduction facts about $\CFI_{\Delta^0_2}$, which point to its strength. We begin with the following. Neumann and Pauly~\cite{NP-2018} introduced the \emph{sorting principle}, $\mathsf{Sort}$, whose instances are all elements of $2^\omega$, such that the instance $p \in 2^\omega$ has the unique solution  $0^n1^\omega$ if $p$ contains exactly $n$ many $0$'s, and $0^\omega$ if $p$ contains infinitely many $0$'s. We refer to~\cite[Definition 1.2]{BGP-TA} for the definitions of the $k$-fold product and the star operation, $*$.

\begin{theorem}
	$\CFI_{\Delta^0_2} \nured \Sort^*$.
\end{theorem}

\begin{proof}
	Assume that $\CFI_{\Delta^0_2} \ured \Sort^*$. Then there is a Turing functional mapping instances of $\CFI_{\Delta^0_2}$ to instances of $\Sort^*$, and hence there are a $\sigma$ and a $k$ such that $\CFI^\sigma_{\Delta^0_2} \ured \Sort^k$. Choose $k$ minimal for which there is such a $\sigma$. By Proposition \ref{prop:restr} we also have $\CFI_{\Delta^0_2} \ured \Sort^k$.
	Let $\Phi$ and $\Psi$ witness the reduction.

        Suppose there is a $\tau$ such that $\Phi^\tau$ outputs at least $n$ many $0$'s for each input to $\Sort$ and the output of $\Psi$ on input $\langle \tau, \langle 0^n, \ldots, 0^n\rangle\rangle$ contains some $l \in \mathbb N$. Then there is some $p$ such that $l \notin \psi(\tau p)$, which is a contradiction. Thus for each $p$, there is some $d \leq k$ such that the $d$th input to $\Sort$ given by $\Phi^p$ has finitely many $0$'s. The set of pairs $\langle d,n \rangle$ such that the $d$th input has some $0$ in a position greater than $n$ is c.e.\ in $p$, so from $p$ we can obtain an instance of $\mathsf{C}_{\mathbb N}$ whose solutions are pairs $\langle d,n \rangle$ such that the $d$th input has no $0$'s at positions greater than $n$. It follows that $\CFI_{\Delta^0_2} \ured \Sort^{k-1} \comp \mathsf{C}_{\mathbb{N}}$. By Proposition \ref{prop:cbabsorption}, this in turn implies $\CFI_{\Delta^0_2} \ured \Sort^{k-1}$, contradicting the minimality of $k$.
\end{proof}

In the next proposition, $\K_{\mathbb{N}}$ denotes the choice problem for \emph{compact} subsets of $\mathbb{N}$ (see~\cite{BGP-TA}). We refer the reader to~\cite[Section 6]{BGP-TA} for the definition of the jump operator, $'$, on Weihrauch degrees. Definitions of the countable coproduct $\coprod$ and the problems $\mathsf{C}_{\{0,\ldots,n\}}$ used in the proof below can also be found in that paper, in Sections 4 and 7, respectively.

\begin{proposition}
	$\CFI_{\Delta^0_2} \ured \mathsf{C}'_{\mathbb{N}}$ but  $\CFI_{\Delta^0_2} \nured \K'_{\mathbb{N}}$.
\end{proposition}

\begin{proof}
  For the reduction, fix some enumeration $(\sigma_i)_{i \in \omega}$ of $\omega^{<\omega}$. Given some input $d$ to $\CFI_{\Delta^0_2}$ we define a sequence $(e_n)_{n \in \omega}$ with $e_n : \omega^2 \to 2$ by $e_n(x,s) = 0$ if{}f for all $y < n$ we have that $d(y,s) = 0$ if{}f $y$ occurs in $\sigma_x$. The sequence $(e_n)_{n \in \omega}$ converges to some $e : \omega^2 \to 2$ with the property that $e(x,s) = 0$ for all $s \in \omega$ precisely when $\sigma_x$ lists exactly those $k$ with $\lim_y d(k,y) = 0$. Clearly, from such a finite tuple we can compute an infinite subset of its complement.

	For the non-reduction, note that $\K'_{\mathbb N} \ured \left (\coprod_{n \in \naturalnumbers} \mathsf{C}_{\{0,\ldots,n\}}' \right ) \comp \mathsf{C}_{\mathbb{N}}$, so if $\CFI_{\Delta^0_2} \ured \K'_{\mathbb N}$, then by Proposition \ref{prop:cbabsorption}, we have $\CFI_{\Delta^0_2} \ured \left (\coprod_{n \in \naturalnumbers} \mathsf{C}_{\{0,\ldots,n\}}' \right )$. As $\CFI_{\Delta^0_2}$ is a fractal (as discussed above), then there is some $k \in \naturalnumbers$ with $\CFI_{\Delta^0_2} \ured \mathsf{C}_{\{0,\ldots,k\}}'$. But this is impossible for reasons of cardinality.
\end{proof}


The \emph{connected choice problem} of the next theorem was introduced by Brattka, Le Roux, Miller, and Pauly~\cite{BRMP-TA}. The instances of $\CC_1$ are nonempty closed subintervals of the real unit interval (see~\cite{BRMP-TA} for details on how the elements of the collection $\mathcal{A}(\uint)$ of such subintervals are represented), and the solutions to any such instance are the points inside it.

\begin{theorem}
	$\CC_1 \nured \CFI_{\Delta^0_2}$.
\end{theorem}

\begin{proof}
	Assume that $\CC_1 \ured \CFI_{\Delta^0_2}$ via $\Phi$ and $\Psi$. Let $p_0 \in \omega^\omega$ be a name for $\uint \in \mathcal{A}(\uint)$. There have to be some finite set $B_0 \subset \naturalnumbers$ and a prefix $\sigma_0$ of $p_0$ such that upon reading $\sigma_0$ and $B_0$, the functional $\Psi$ outputs a $2^{-2}$-approximation of some $x_0 \in \uint$. We can find some $\tau_0p_1$ such that $\sigma_0\tau_0p_1$ is a name for some interval $I_1$ with $|I_1| \geq 2^{-2}$ and such that for any $q$ extending $\sigma_0\tau_0$ and representing some $A \in \mathcal{A}(\uint)$, we have that $A \cap B(x_0,2^{-2}) = \emptyset$ (where $B(x_0,2^{-2})$ is the ball of radius $2^{-2}$ around $x_0$). It follows that for any $q$ extending $\sigma_0\tau_0$, the set $\psi(\Phi^q)$ must not contain $B_0$, for if it did, there would be a solution to $\Phi^q$ containing $B_0$ that would trick $\Psi$ into outputting a $2^{-2}$-approximation of $x_0$, which cannot be correct.
	
	In the next step, $\Psi$ has to output some $2^{-4}$-approximation of some $x_1$ upon reading some prefix $\sigma_0 \tau_0 \sigma_1$ of $\sigma_0 \tau_0 p_1$ and a finite set $B_1$ with $\max B_1 > \min B_0$. We pick $\tau_1$ to exclude $B(x_1,2^{-4})$ from the solution set, and thus conclude that for any $q$ extending $\sigma_0 \tau_0 \sigma_1 \tau_1$, the set $\psi(\Phi^q)$ must not contain $B_1$ (nor $B_0$).
	
	By iterating the procedure, we obtain some input $\sigma_0 \tau_0 \sigma_1 \tau_1 \sigma_2 \tau_2 \cdots \in \omega^\omega$, which is in the domain of $\CC_1$ (as this has a total domain if represented in a suitable way), but such that $\psi(\Phi^{\sigma_0 \tau_0 \sigma_1 \tau_1 \sigma_2 \tau_2 \cdots})$ excludes countably many disjoint finite sets $B_0, B_1, \ldots$. Hence, $\psi(\Phi^{\sigma_0 \tau_0 \sigma_1 \tau_1 \sigma_2 \tau_2 \cdots}) \notin \dom(\CFI_{\Delta^0_2})$, and we have derived a contradiction.
\end{proof}

Brattka, H{\"o}lzl, and Kuyper~\cite[Proposition 16]{hoelzl2} showed that $\CC_1 \ured \Sort$, so it follows that $\Sort \nured \CFI_{\Delta^0_2}$. An alternate proof of this fact can be given by using the following technical notion.

\begin{definition}
	Suppose $G : \mathop{\subseteq} \mathbf{X} \rightrightarrows \mathbf{Z}$ is a partial multifunction of represented spaces. Then $G$ is \emph{low for functions} if, for every $f : \mathbf{Y} \to \Baire$ that satisfies $f \ured \limprob \comp G$, we have $f \ured \limprob$.
\end{definition}

\begin{proposition}\label{prop:lowness}
	Let $G : \mathbf{X} \mto \mathcal{O}(\naturalnumbers)$ (where $\mathcal{O}(\naturalnumbers)$ consists of the subsets of $\omega$ represented by enumerations of their elements) be such that \[(\forall x \in \mathbf{X})(\exists k_0 \in \naturalnumbers)(\forall k \geq k_0)[\{n : n \geq k\} \in G(x)].\]
	Then $G$ is low for functions.
\end{proposition}

\begin{proof}
Let $f : \mathbf{Y} \to \Baire$ be such that $f \ured \limprob \comp G$.
	Without loss of generality, we may assume that $\mathbf{X}, \mathbf{Y} \subseteq \Baire$. As $\limprob$ is transparent (see Brattka, Gherardi, and Marcone~\cite[Fact 5.5]{BGM-2012}), we can obtain $f(x) = \lim_{i \to \infty} \Psi_i(x,G(\Phi(x)))$ for some functionals $\Phi$ and $\Psi_i$. Let $\naturalnumbers_{\geq k} = \{n \in \naturalnumbers : n \geq k\} \in \mathcal{O}(\naturalnumbers)$. Now for any $x \in \mathbf{X}$, we have that $\Psi_k(x, \naturalnumbers_{\geq k})$ is defined and is an element of $\Psi_k(x,G(\Phi(x)))$ for almost all $k$. As $f$ is a function, in $\lim_{i \to \infty} \Psi_i(x,G(\Phi(x)))$ it does not matter whether we choose from $G(\Phi(x))$ once for the entire expression, or separately for each $i$. Thus, we can compute $f(x)$ as $\lim_{k \to \infty} \Psi_k(x, \naturalnumbers_{\geq k})$. (While finitely many of these values may be undefined, this problem can be resolved with a standard argument.)
\end{proof}

\begin{lemma}\label{obs:lowness}
	Let $G$ be low for functions, and $f : \mathbf{X} \to \Baire$ with $f \ured G$. Then $\limprob \comp f \ured \limprob$.
\end{lemma}

\begin{proof}
As $f$ is a function, so is $\limprob \comp f$. Moreover, $f \ured G$ implies $\limprob \comp f \ured \limprob \comp G$, so $\limprob \comp f \ured \limprob$.
\end{proof}

\begin{corollary}
$\Sort \nured \CFI_{\Delta^0_2}$.
\end{corollary}

\begin{proof}
By Proposition \ref{prop:lowness}, $\CFI_{\Delta^0_2}$ is low for functions. As $\Sort$ is a function, Lemma \ref{obs:lowness} shows that if we had $\Sort \ured \CFI_{\Delta^0_2}$ we would also have $\limprob \comp \Sort \ured \limprob$. However, it is not difficult to check that $\lpo' \ured \lpo \comp \Sort \ured \limprob \comp \Sort$, but $\lpo' \nured \limprob$. (That $\limprob \comp \Sort \nured \limprob$ also follows from~\cite[Proposition 21]{hoelzl2}.) To see that this non-reduction holds, first note that there is a uniformly computable sequence $S_0,
S_1, \ldots$ of instances of $\lpo'$ such that for each $e$, the $e$th
Turing functional $\Phi_e$ is total if and only if $S_e = 0^\omega$. Thus,
for each $e$, to determine whether $S_e$ has solution $0$ is $\Pi^0_2$-hard.
On the other hand, every computable $\limprob$-instance has a uniformly
$\Delta^0_2$ solution.
\end{proof}

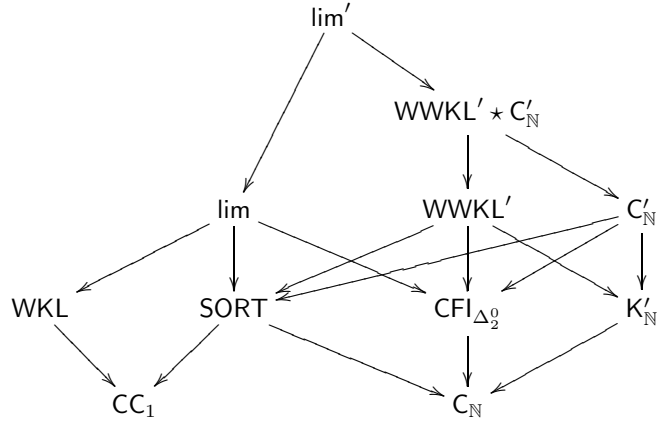
\begin{figure}[t]
\[
	\xymatrix@C=1em@R=2em{
& & & \limprob'  \ar@{->}[dr] \ar@{->}[ddl]\\
& & & & \WWKL' \comp \mathsf{C}_{\mathbb{N}}' \ar@{->}[d] \ar@{->}[drr] \\
& & \limprob \ar@{->}[d] \ar@{->}[dll] \ar@{->}[drr] & &  \WWKL' \ar@{->}[d] \ar@{->}[dll] \ar@{->}[drr] & & \mathsf{C}_{\mathbb{N}}' \ar@{->}[d] \ar@{->}[dll] \ar@{->}[dllll]\\
\WKL \ar@{->}[dr] & &  \mathsf{SORT}  \ar@{->}[dl] \ar@{->}[drr] & & \CFI_{\Delta^0_2} \ar@{->}[d] & & \mathsf{K}_{\mathbb{N}}' \ar@{->}[dll]\\
& \mathsf{CC}_1 & & &  \mathsf{C}_{\mathbb{N}}
}
\]
	\caption[]{Location of $\CFI_{\Delta^0_2}$ in the Weihrauch lattice. An arrow from $\mathsf{P}$ to $\mathsf{Q}$ represents the reduction $\mathsf{Q} \ured \mathsf{P}$. No additional arrows can be added other than those that follow by transitivity. For (non-)reductions not explicitly mentioned above, see \cite[Sections 4 and 5 and Figure 5]{hoelzl2}.}\label{F:CFI}
\end{figure}

Note that while the proof above shows that $\limprob \comp \Sort \nured
\limprob$, it was shown by Neumann and Pauly~\cite[Corollary 32]{NP-2018}
that $\limprob \comp \limprob \comp \Sort \ured \limprob \comp \limprob$.

We conclude with one final reduction. Recall that $\WWKL$ is the problem whose instances are closed subsets of $2^\omega$ of positive measure, with solutions being the members of the given set. We refer the reader to Downey and Hirschfeldt~\cite[Chapter 6]{DH-2010} for background on Martin-L\"{o}f randomness. 

\begin{proposition}
	$\CFI_{\Delta^0_2} \ured \WWKL'$.
\end{proposition}

\begin{proof}
	Let $\mathcal{C}_0$ be a fixed $\Pi^{0,\emptyset'}_1$ class all of whose members are $2$-random. Given an instance $c : [\omega]^2 \to 2$ of $\CFI_{\Delta^0_2}$, let $\mathcal{C}$ consist of all $X \in \mathcal{C}_0$ such that for all $i$, if $X(i) = 1$ then $\lim_s c(i,s) = 1$. Then $\mathcal{C}$ is a $\Pi^{0,c'}_1$ subclass of $\mathcal{C}_0$, and it still has positive measure since, e.g., if $X$ is any $2$-random real and $n$ is least such that $c(i,s) = 1$ for all $i \geq n$, then $\mathcal{C}$ contains the $2$-random real $0^{n-1}X(n)X(n+1) \cdots$. Thus, $\mathcal{C}$ may be regarded as an instance of $\WWKL'$, and if $X$ is any element of $\mathcal{C}$ then $\{i : X(i) = 1\}$ is infinite and is therefore a $\CFI_{\Delta^0_2}$-solution to $c$.
	\end{proof}

We summarize the results of this section in Figure \ref{F:CFI}.

\section{Ramsey's theorem for singletons}

In this section, we investigate Ramsey's theorem for singletons and different numbers of colors, and how these problems behave under Weihrauch reducibility with respect to products. A motivating toy example is the fact that $\RT^1_2 \times \RT^1_2 \ured \RT^1_4$, and in fact, it is easy to see that for all $n \geq 1$ and $k_0,\dots,k_n \geq 2$,
\[ \prod_{m=0}^n \RT^1_{k_m} \sured \RT^1_{\prod_{m=0}^n k_m}. \]
We show below that the right-hand side is optimal. Our results extend a number of similar investigations, including by Dorais, Dzhafarov, Hirst, Mileti, and Shafer~\cite{DDHMS-2016}, Hirschfeldt and Jockusch~\cite{HJ-2016}, Patey~\cite{Patey-2016}, and Brattka and Rakotoniaina~\cite{BR-2017}.

In the sequel, we will regard $\RT^1_k$ as the problem whose instances are colorings $c: \omega \to k$ and whose solutions are colors that appear infinitely often in $c$. Note that this formulation of $\RT^1_k$ is Weihrauch equivalent to the more usual one given in Definition \ref{def:ramsey}, so we will not distinguish these versions when discussing Weihrauch reducibility. In the context of strong Weihrauch reducibility, we will refer to the new version as $\RTc^1_k$. The principle $\RTc^1_k$ can be understood as the Bolzano-Weierstrass theorem for the discrete space $k$, and was indeed studied as $\mathsf{BWT}_k$ by Brattka, Gherardi and Marcone~\cite{BGM-2012}. A central result there is that $\RTc^1_k \suequiv \C_k'$, which tells us that we could alternatively strive to understand the principles $\RTc^1_k$ by studying the finite choice principles $\C_k$, and transferring the results using the jump of strong Weihrauch degrees.\footnote{This approach seems very promising, but is left to future work.}

Given this formulation, the backward functionals of our strong Weihrauch reductions will have single numbers or tuples of numbers as oracles, and hence can be regarded as partial functions. For such a functional $\Psi$, we write $\Psi(n)$ instead of $\Psi^n$.


We begin with the following lemma:

\begin{lemma} \label{lem:W_implies_sW_finite_tolerance}
Suppose that $\mathsf{P} \ured \mathsf{Q}$ and these problems satisfy the following properties:
\begin{itemize}
	\item $\mathsf{P}$ has finite tolerance, i.e., there is some $\Theta$ such that if $C_0$ and $C_1$ are $\mathsf{P}$-instances, $C_0(x) = C_1(x)$ for all $x$ above some $m$, and $S_0$ is a $\mathsf{P}$-solution to $C_0$, then $\Theta^{S_0 \oplus m}$ is a $\mathsf{P}$-solution to $C_1$;
	\item any finite modification of a $\mathsf{P}$-instance is still a $\mathsf{P}$-instance;
	\item solutions to all instances of $\mathsf{P}$ and $\mathsf{Q}$ lie in some fixed finite set.
\end{itemize}
Then $\mathsf{P} \sured \mathsf{Q}$.
\end{lemma}
\begin{proof}
Fix functionals $\Phi$ and $\Psi$ witnessing that $\mathsf{P} \ured \mathsf{Q}$. Since solutions to all instances of $\mathsf{P}$ lie in some fixed finite set, we may assume that for each $\mathsf{P}$-instance $C$ and each $s$ that is a $\mathsf{Q}$-solution to $\Phi^C$, we have that $\Psi^{C \oplus s}$ outputs a number that codes a $\mathsf{P}$-solution to $C$. Fix a functional $\Theta$ witnessing that $\mathsf{P}$ has finite tolerance. Fix a finite solution set $S$ for $\mathsf{Q}$. We define functionals that witness that $\mathsf{P} \sured \mathsf{Q}$.

First, we construct a $\tau$ that is a finite initial segment of some $\mathsf{P}$-instance, such that $\tau$ decides (in the sense of Cohen $1$-genericity) for each $s \in S$ whether $\Psi^{C \oplus s}$ converges for $\mathsf{P}$-instances $C$ extending $\tau$. Since $S$ is finite, such a $\tau$ exists.

We define $\widehat{\Phi}$ by $\widehat{\Phi}^C = \Phi^{C'}$, where $C'$ is obtained from $C$ by replacing its initial segment of length $|\tau|$ by $\tau$ itself. By our assumption on $\mathsf{P}$, this $C'$ is still a $\mathsf{P}$-instance.

We define $\widehat{\Psi}$ by $\widehat{\Psi}(s) = \Theta^{\Psi^{\tau \oplus s} \oplus |\tau|}$. We show that $\widehat{\Phi}$ and $\widehat{\Psi}$ witness that $\mathsf{P} \sured \mathsf{Q}$.

Take any $\mathsf{P}$-instance $C$. Since $C'$ is a $\mathsf{P}$-instance, $\widehat{\Phi}^C = \Phi^{C'}$ is a $\mathsf{Q}$-instance. Let $s$ be any $\mathsf{Q}$-solution to $\Phi^{C'}$. Then $\Psi^{C' \oplus s}$ is a $\mathsf{P}$-solution to $C'$. In particular, $\Psi^{C' \oplus s}$ converges. Since $C'$ extends $\tau$, by our construction of $\tau$, we have that $\Psi^{\tau \oplus s}\converges = \Psi^{C' \oplus s}\converges$. Hence $\Psi^{\tau \oplus s}$ is a $\mathsf{P}$-solution to $C'$. We conclude that $\widehat{\Psi}(s) = \Theta^{\Psi^{\tau \oplus s} \oplus |\tau|}$ is a $\mathsf{P}$-solution to $C$.
\end{proof}

It is easy to see that $\RTc^1_k$ (and finite parallel products of $\RTc^1_k$) satisfy the properties of $\mathsf{P}$ and $\mathsf{Q}$ in Lemma \ref{lem:W_implies_sW_finite_tolerance}. Therefore we have the following.

\begin{corollary}
If $\prod_{m=0}^n \RT^1_{k_m} \ured \RT^1_N$, then $\prod_{m=0}^n \RTc^1_{k_m} \sured \RTc^1_N$.
\end{corollary}

Optimality then follows from a counting argument:

\begin{proposition}
If $\prod_{m=0}^n \RTc^1_{k_m} \sured \RTc^1_N$, then $N \geq \prod_{m=0}^n k_m$.
\end{proposition}
\begin{proof}
Fix $\Phi$ and $\Psi$ witnessing that $\prod_{m=0}^n \RTc^1_{k_m} \sured \RTc^1_N$. We show that for each $(a_0,\dots,a_n) \in \prod_{m=0}^n k_m$, there is some $i < N$ such that $\Psi(i) = (a_0,\dots,a_n)$.

Consider the tuple of constant colorings $(a_0^\omega,\dots,a_n^\omega)$. This is a $\prod_{m=0}^n \RTc^1_{k_m}$-instance, so $\Phi^{(a_0^\omega,\dots,a_n^\omega)}$ is an $\RTc^1_N$-instance with some solution $i$. Then $\Psi(i)$ must be a solution to $(a_0^\omega,\dots,a_n^\omega)$, so $\Psi(i) = (a_0,\dots,a_n)$.
\end{proof}

\begin{corollary}
If $\prod_{m=0}^n \RT^1_{k_n} \ured \RT^1_N$, then $N \geq \prod_{m=0}^n k_m$.
\end{corollary}

Therefore the right-hand side of $\prod_{m=0}^n \RTc^1_{k_m} \sured \RTc^1_{\prod_{m=0}^n k_m}$ is optimal, with regards to both $\ured$ and $\sured$. However, we will see that $\RT^1_{\prod_{m=0}^n k_m} \nured \prod_{m=0}^n \RT^1_{k_m}$ for all $n \geq 1$ and $k_0,\dots,k_n \geq 2$ (Proposition \ref{prop:lousier_bound}). In the rest of this section, we attempt to find the smallest $N$ such that
\[ \RT^1_N \nured \prod_{m=0}^n \RT^1_{k_m}. \]

We start by giving a lower bound for $N$.

\begin{proposition} \label{prop:lower_bound}
For all $n \geq 1$ and $k_0,\dots,k_n \geq 2$,
\[ \RTc^1_{1+\sum_{m=0}^n (k_m-1)} \sured \prod_{m=0}^n \RTc^1_{k_m}. \]
\end{proposition}
\begin{proof}
Suppose we are given an instance $c$ of $\RTc^1_{1+\sum_{m=0}^n (k_m-1)}$. For $0 \leq m \leq n$, we define colorings
\[ d_m: \omega \to \left\{\sum_{i=0}^{m-1} (k_i-1),\dots,\sum_{i=0}^m (k_i-1)\right\} \]
as follows. Note that for each $m$, $d_m$ will be a $k_m$-coloring.

For each $m$ and $x$, we define $d_m(x)$ as follows. First check which color among $0,\dots,\sum_{i=0}^m (k_i-1)$ appears most often among $c(0),\dots,c(x)$. (Resolve ties by picking the smallest color.) If this color is among $0,\dots,\sum_{i=0}^{m-1} (k_i-1)$, let $d_m(x) = \sum_{i=0}^{m-1} (k_i-1)$. Otherwise, let $d_m(x)$ be this color.

Now, given $(a_0,\ldots,a_n)$ such that, for each $m$, the color $a_m$ appears infinitely often in $d_m$, we want to compute a color that appears infinitely often in $c$. Start by considering $a_n$. If $a_n \neq \sum_{i=0}^{n-1} (k_i-1)$, then for infinitely many $x$, the color $a_n$ appears most often among $c(0),\dots,c(x)$. In particular, $a_n$ appears infinitely often in $c$.

On the other hand, if $a_n = \sum_{i=0}^{n-1} (k_i-1)$, then for infinitely many $x$, some color among $0,\dots,\sum_{i=0}^{n-1} (k_i-1)$ appears most often among $c(0),\dots,c(x)$. By the pigeonhole principle, some color among $0,\dots,\sum_{i=0}^{n-1} (k_i-1)$ appears infinitely often in $c$. We then proceed to consider $a_{n-1}$ and repeat the above case division. Eventually we either reach some $a_m$ that is not equal to $\sum_{i=0}^{m-1} (k_i-1)$, in which case $a_m$ appears infinitely often in $c$, or we reach $a_0 = 0$, in which case $0$ appears infinitely often in $c$.
\end{proof}


In order to obtain upper bounds for $N$, we begin by restricting the reductions that we need to diagonalize against. Firstly, by Lemma \ref{lem:W_implies_sW_finite_tolerance}, we need only handle strong Weihrauch reductions:

\begin{proposition}
If $\RT^1_N \ured \prod_{m=0}^n \RT^1_{k_m}$, then $\RTc^1_N \sured \prod_{m=0}^n \RTc^1_{k_m}$.
\end{proposition}

We can impose a further restriction:

\begin{lemma}
Suppose $\RTc^1_N \sured \prod_{m=0}^n \RTc^1_{k_m}$ via some forward functionals $\Phi_m$, $0 \leq m \leq n$, where $\Phi_m$ computes the $m^{\text{th}}$ coloring in the $\prod_{m=0}^n \RTc^1_{k_m}$-instance, and a backward functional $\Psi$. Then for any $i < N$, there exists $(a_0,\dots,a_n)$ where each $a_m < k_m$ and $\Psi(a_0,\dots,a_n) = i$.
\end{lemma}
\begin{proof}
Given $i < N$, consider the coloring $c$ that is constantly $i$. Then the tuple $(\Phi^c_0,\dots,\Phi^c_n)$ is a $\prod_{m=0}^n \RTc^1_{k_m}$-instance. Hence it has some solution $(a_0,\dots,a_n)$. The only solution to $c$ is $i$, so $\Psi(a_0,\dots,a_n)$ must be $i$.
\end{proof}

Combining the previous two facts, we obtain:

\begin{corollary} \label{cor:reduction_wlog}
Suppose $\RT^1_N \ured \prod_{m=0}^n \RT^1_{k_m}$. Then $\RTc^1_N \sured \prod_{m=0}^n \RTc^1_{k_m}$, as witnessed by some $\Phi_m$, $0 \leq m \leq n$, and $\Psi$ where $\Psi: \prod_{m=0}^n k_m \to N$ is a surjective partial function.
\end{corollary}

Henceforth, we will always assume that our reductions of $\RTc^1_N$ to $\prod_{m=0}^n \RTc^1_{k_m}$ have the above special form. In order to diagonalize against such reductions, it will be convenient to have the following notion of covering a tuple of colors using a set of tuples of colors.

\begin{definition}
If $X \subseteq \prod_{m=0}^n k_m$ and $(i_0,\dots,i_n) \in \prod_{m=0}^n k_m$, we say that $X$ \emph{covers} $(i_0,\dots,i_n)$ if for each $0 \leq m \leq n$, there is an $(a_0,\dots,a_n) \in X$ such that $a_m = i_m$.
\end{definition}

Observe that if $c$ is a $\prod_{m=0}^n \RTc^1_{k_m}$-instance whose solution set contains $X$, and $X$ covers $(i_0,\dots,i_n)$, then $(i_0,\dots,i_n)$ is also a solution to $c$.

The following terminology will also be useful.

\begin{definition}
For a surjective partial function $\Psi: \prod_{m=0}^n k_m \to N$, we refer to each $\Psi^{-1}(i)$ as a \emph{fiber}. We call a fiber of size one a \emph{singleton}.
\end{definition}

We now work towards an upper bound ($\approx \frac{\prod k_m}{2}$) for $N$. Suppose we want to show that $\RT^1_N \nured \prod_{m=0}^n \RT^1_{k_m}$ for some $N$. Towards a contradiction, we may (by Corollary \ref{cor:reduction_wlog}) fix $\Phi_m$, $0 \leq m \leq n$, and $\Psi$ witnessing that $\RTc^1_N \sured \prod_m \RTc^1_{k_m}$ such that $\Psi$ is a surjective partial function from $\prod_{m=0}^n k_m$ to $N$. We aim to construct $c: \omega \to N$ and some $(a_0,\dots,a_n)$ such that $(a_0,\dots,a_n)$ is a solution to $\Phi^c_0,\dots,\Phi^c_n$, yet $\Psi(a_0,\dots,a_n)$ is not a solution to $c$.

Our basic strategy is to choose $N$ large enough so that the following combinatorial property holds for all surjective partial functions $\Psi: \prod_{m=0}^n k_m \to N$:

\vspace{2ex}\par
\hfill\parbox{\dimexpr 0.75\textwidth}
{There is some nonempty $S \subsetneq N$ such that for any set of $(a_0,\dots,a_n)$'s whose image under $\Psi$ is exactly $S$, the $(a_0,\dots,a_n)$'s cover some $(b_0,\dots,b_n)$ that maps outside $S$ under $\Psi$.}
\hfill\llap{($\ast$)}\vspace{2ex}\par

Assuming ($\ast$), we may construct $c$ by repeatedly looping through colors in $S$: for each $i \in S$, extend constantly by $i$ until there is some $(a_0,\dots,a_n)$ that maps to $i$ under $\Psi$, such that for all $0 \leq m \leq n$, we have that $\Phi^c_m$ has some new element of color $a_m$. (This must happen eventually: if $c$ is the $\RTc^1_N$-instance produced by extending the current finite coloring by $i$ forever, then $\Phi^c_0,\dots,\Phi^c_n$ is a $\prod_{m=0}^n \RTc^1_{k_m}$-instance with some solution $(a_0,\dots,a_n)$. Then $\Psi(a_0,\dots,a_n) = i$, and for each $0 \leq m \leq n$, some new element of color $a_m$ must appear at some finite stage of $\Phi^c_m$.)

Then for each $i \in S$, there is some $(a_0,\dots,a_n)$ such that $\Psi(a_0,\dots,a_n) = i$ and $(a_0,\dots,a_n)$ is a solution to $\Phi^c_0,\dots,\Phi^c_n$. But then the $(a_0,\dots,a_n)$'s cover some $(b_0,\dots,b_n)$ that maps outside $S$ under $\Psi$. It follows that $(b_0,\dots,b_n)$ is also a solution to $\Phi^c_0,\dots,\Phi^c_n$. But $\Psi(b_0,\dots,b_n) \notin S$ and is hence not a solution to $c$, which is a contradiction. Thus $\RTc^1_N \nsured \prod_{m=0}^n \RTc^1_{k_m}$, and hence $\RT^1_N \nured \prod_{m=0}^n \RT^1_{k_m}$.

The above strategy may be applied as follows:

\begin{proposition} \label{prop:k=1_upper_bound}
If $N > \max\{\frac{k_0 \cdot k_1}{2},k_0+k_1-1\}$, then $\RT^1_N \nured \RT^1_{k_0} \times \RT^1_{k_1}$.
\end{proposition}
\begin{proof}
By the previous discussion, it suffices to show that ($\ast$) holds. Since $N > \frac{k_0 \cdot k_1}{2}$, by a counting argument, $\Psi$ must have at least one singleton $(a_0,a_1)$. Note that there are $1+(k_0-1)+(k_1-1) = k_0+k_1-1$ many pairs in $k_0 \times k_1$ that share some color with $(a_0,a_1)$. 	But $N > k_0+k_1-1$, so there is some fiber $G$ such that none of its pairs share any colors with $(a_0,a_1)$. In other words, for every pair in $G$, the set containing it and $(a_0,a_1)$ covers a pair outside $G$. Let $S$ be the image of $(a_0,a_1)$ and $G$ under $\Psi$. Then $S$ witnesses that ($\ast$) holds.
\end{proof}

\begin{corollary} \label{cor:cases_up_until_8}
We have that
\begin{align*}
\RT^1_4 &\nured \RT^1_2 \times \RT^1_2, &\RT^1_5 \nured \RT^1_2 \times \RT^1_3, \\
\RT^1_6 &\nured \RT^1_2 \times \RT^1_4, &\RT^1_6 \nured \RT^1_3 \times \RT^1_3, \\
\RT^1_7 &\nured \RT^1_2 \times \RT^1_5, &\RT^1_7 \nured \RT^1_3 \times \RT^1_4, \\
\RT^1_8 &\nured \RT^1_2 \times \RT^1_6, &\RT^1_8 \nured \RT^1_3 \times \RT^1_5.
\end{align*}
\end{corollary}

Note that Proposition \ref{prop:lower_bound} implies that $\RT^1_{k_0+k_1-1} \ured \RT^1_{k_0} \times \RT^1_{k_1}$. Hence all of the non-reductions in Corollary \ref{cor:cases_up_until_8} are sharp. We will address the missing case of $\RT^1_8$ and $\RT^1_4 \times \RT^1_4$ in Proposition \ref{44prop}.

We can derive more results using variations of the argument in Proposition \ref{prop:k=1_upper_bound}.

\begin{proposition} \label{prop:lousier_bound}
If \[N > \frac{\max k_m + \prod_{m=0}^n k_m}{2},\] then $\RT^1_N \nured \prod_{m=0}^n \RT^1_{k_m}$.
\end{proposition}
\begin{proof}
As before, we show that ($\ast$) holds. By a counting argument, $\Psi$ must have at least $1+\max k_m$ many $(a_0,\dots,a_n)$ that are singletons. Among these singletons, there must be two of them that differ in at least two entries, i.e., the set consisting of these two singletons covers a new tuple of colors. We can then take $S$ to be the image of two such singletons under $\Psi$.
\end{proof}

We can improve on this bound asymptotically, but even then this result seems to be far from optimal.

\begin{proposition} \label{prop:better_bound}
If \[N > \max\left\{\frac{2+\prod_{m=0}^n k_m}{2},\ \max k_m - 1 + \frac{\prod_{m=0}^n k_m}{3}\right\},\] then $\RT^1_N \nured \prod_{m=0}^n \RT^1_{k_m}$.
\end{proposition}
\begin{proof}
As before, we show that ($\ast$) holds. Since $N > \frac{2+\prod_{m=0}^n k_m}{2}$, the reduction $\Psi$ must have at least three singletons.

\underline{Case 1.} If there are two singletons that differ in at least two entries, then we may take $S$ to be the image of two such singletons under $\Psi$, as in Proposition \ref{prop:lousier_bound}.

\underline{Case 2.} Otherwise, all of the singletons share exactly one common entry. So there are some $0 \leq m \leq n$ and $3 \leq l \leq k_m$ such that there are exactly $l$ many singletons and all of them are of the form $(a_0,\dots,a_{m-1},b,a_{m+1},\dots,a_n)$, where $b < k_m$.

We claim that there are at least $k_m+1$ many fibers of size $< l$. If not, by a counting argument, there are at least
\begin{align*}
&1 \cdot l + 2 \cdot (k_m-l) + l \cdot (N-k_m) \\
= \; &l + 2k_m - 2l + lN - lk_m \\
> \; &l\left(\max k_m - 1 + \frac{\prod_{m=0}^n k_m}{3}\right) + 2k_m - l - lk_m \\
\geq \; &lk_m - l + \prod_{m=0}^n k_m + 2k_m - l - lk_m \\
\geq \; &\prod_{m=0}^n k_m
\end{align*}
many tuples, which is a contradiction.

By the claim, there is a fiber $U$ of size $< l$ that does not contain any tuple of the form $(a_0,\dots,a_{m-1},b,a_{m+1},\dots,a_n)$. Since $|U| < l$, there is a singleton $(a_0,\dots,a_{m-1},b,a_{m+1},\dots,a_n)$ such that $b$ does not appear in any tuple in $U$. Then for any tuple in $U$, the set containing it and $(a_0,\dots,a_{m-1},b,a_{m+1},\dots,a_n)$ covers some tuple outside $U$, so we can take $S$ to be the image of $U$ and said singleton.
\end{proof}

The lower bound in Proposition \ref{prop:lower_bound} is, in general, much smaller than the upper bounds in Propositions \ref{prop:k=1_upper_bound}, \ref{prop:lousier_bound}, and \ref{prop:better_bound}. Observe that in all of our proofs, the sets $S$ consist of two elements, at least one of which is the image of a singleton under $\Psi$. However, $\Psi$ may not have any singletons, for example in a hypothetical reduction witnessing that $\RT^1_8 \ured \RT^1_4 \times \RT^1_4$. Also, there may not be any $S$ that has exactly two elements and satisfies ($\ast$), e.g., consider $\Psi: 4 \times 4 \to 8$ as represented in the grid below. Here $\Psi$ maps $(i,j) \in 4 \times 4$ to the number in the $(i,j)^{\text{th}}$ position.
\[ \begin{matrix}
0 & 3 & 2 & 6 \\
0 & 4 & 5 & 7 \\
1 & 2 & 3 & 7 \\
1 & 4 & 5 & 6
\end{matrix} \]
One can check that for any $c,d < 8$, there is a point labeled $c$ that shares a row or column with a point labeled $d$. That means that $S = \{c,d\}$ fails to satisfy ($\ast$).

Therefore, new techniques will be required to close the gap between our lower and upper bounds. We conclude this section by giving an ad hoc proof that $\RT^1_8 \nured \RT^1_4 \times \RT^1_4$, which is the smallest case not resolved by Corollary \ref{cor:cases_up_until_8}. In order to do so, we will show that there exists some $S$ that satisfies ($\ast$) and has exactly \emph{three} elements.

Before specializing to the case of $\RT^1_8 \nured \RT^1_4 \times \RT^1_4$, we consider a more general context: let $k_0,k_1 \geq 2$ and fix a surjective partial function $\Psi: k_0 \times k_1 \to N$ (i.e., a potential backward reduction for $\RTc^1_N \sured \RTc^1_{k_0} \times \RTc^1_{k_1}$). We say that a collection of three fibers is \emph{bad} if its image under $\Psi$ does not satisfy ($\ast$). We can characterize the bad collections of three fibers:

\begin{lemma} \label{lem:bad_three_fiber_characterization}
Let $k_0,k_1 \geq 2$ and let $\Psi: k_0 \times k_1 \to N$ be a surjective partial function. A collection of three fibers is bad if and only if their union contains either:
\begin{enumerate}
	\item three pairs in a row/column (e.g., $(a,b_0)$, $(a,b_1)$, $(a,b_2$)), with one pair from each of the three fibers;
	\item four pairs that form a rectangle (i.e., $(a_0,b_0)$, $(a_0,b_1)$, $(a_1,b_0)$, $(a_1,b_1)$), with at least one pair from each of the three fibers.
\end{enumerate}
\end{lemma}
\begin{proof}
($\Leftarrow$). If (1) holds, the three pairs in question do not cover any new pair. If (2) holds, pick three out of the four pairs such that one pair from each of the three fibers is picked. Then these three pairs cover exactly one other pair (the fourth). But the fourth pair is already contained in the union of the three fibers.

($\Rightarrow$). Suppose that we have a bad collection of three fibers. Without loss of generality, we may pick one pair $(a_i,b_i)$ from each fiber such that the three pairs $(a_0,b_0)$, $(a_1,b_1)$, and $(a_2,b_2)$ witness badness.

\underline{Case 1.} $(a_0,b_0)$, $(a_1,b_1)$, and $(a_2,b_2)$ lie in the same row or column. Then they satisfy (1).

\underline{Case 2.} Two out of the three pairs, say $(a_0,b_0)$ and $(a_1,b_1)$, lie in the same row or column (i.e., $a_0 = a_1$ or $b_0 = b_1$). Without loss of generality, suppose that $b_0 = b_1$. Note that $(a_0,b_0)$, $(a_0,b_1)$, and $(a_2,b_2)$ cover $(a_2,b_0)$, $(a_0,b_2)$, and $(a_2,b_1)$. Therefore by badness, the latter three pairs lie in the union of the three fibers.

If $(a_0,b_0)$, $(a_1,b_1)$, and $(a_2,b_2)$ are vertices of a rectangle (i.e., $b_2 = b_0$ or $b_2 = b_1$), then we satisfy (2). Otherwise, we consider cases depending on which fiber contains $(a_2,b_0)$. In all cases, we satisfy either (1) or (2). See Figure \ref{fig:bad_three_fiber_characterization_case_2} for an illustration.


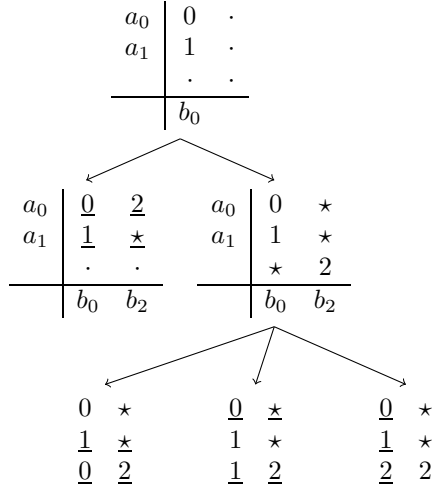
\begin{figure}
\centering
\begin{tikzpicture}
	\node (2) at (0,0) {
$\begin{array}{c|c c}
a_0 & 0 & \cdot \\
a_1 & 1 & \cdot \\
 & \cdot & \cdot \\
\hline & b_0 &
\end{array}$};
	\node (2a) at (-1.25,-2.5) {
$\begin{array}{c|cc}
a_0 & \underline{0} & \underline{2} \\
a_1 & \underline{1} & \underline{\star} \\
& \cdot & \cdot \\
\hline & b_0 & b_2
\end{array}$};
	\draw[->] (2.south) -- (2a.north);
	\node (2b) at (1.25,-2.5) {
$\begin{array}{c|cc}
a_0 & 0 & \star \\
a_1 & 1 & \star \\
& \star & 2 \\
\hline & b_0 & b_2
\end{array}$};
	\draw[->] (2.south) -- (2b.north);
	\node (2ba) at (-1,-5) {
$\begin{matrix}
0 & \star \\
\underline{1} & \underline{\star} \\
\underline{0} & \underline{2}
\end{matrix}$};
	\draw[->] (2b.south) -- (2ba.north);
	\node (2bb) at (1,-5) {
$\begin{matrix}
\underline{0} & \underline{\star} \\
1 & \star \\
\underline{1} & \underline{2}
\end{matrix}$};
	\draw[->] (2b.south) -- (2bb.north);
	\node (2bc) at (3,-5) {
$\begin{matrix}
\underline{0} & \star \\
\underline{1} & \star \\
\underline{2} & 2
\end{matrix}$};
	\draw[->] (2b.south) -- (2bc.north);
\end{tikzpicture}
\caption{Case 2 in Lemma \ref{lem:bad_three_fiber_characterization}, assuming that $b_0 = b_1$. In the array on the top level, $0$ lies in position $(a_0,b_0)$ and $1$ lies in position $(a_1,b_0)$, meaning that $\Psi(a_0,b_0) = 0$ and $\Psi(a_1,b_0) = 1$. We have yet to label position $(a_2,b_2)$. The middle level represents cases depending on whether $a_2$ equals some $a_i$, or not. If a star lies in position $(a,b)$, then $(a,b)$ is known (by badness) to lie in the union of the bad collection of three fibers. Sets of pairs that satisfy (1) or (2) are underlined. The bottom level represents cases depending on which of the three fibers contains $(a_2,b_0)$. For example, in the array on the bottom right, $2$ lies in positions $(a_2,b_0)$ and $(a_2,b_2)$, meaning that $\Psi(a_2,b_0) = \Psi(a_2,b_2) = 2$ and hence $(a_2,b_0)$ and $(a_2,b_2)$ lie in the same fiber. Then $(a_0,b_0)$, $(a_1,b_0)$, and $(a_2,b_0)$ lie in a column, satisfying (1).}
\label{fig:bad_three_fiber_characterization_case_2}
\end{figure}

\underline{Case 3.} None of the three pairs lie in the same row or column. Note that by badness, $(a_0,b_1)$, $(a_1,b_0)$, $(a_0,b_2)$, $(a_2,b_0)$, $(a_1,b_2)$, and $(a_2,b_1)$ all lie in the union of the three fibers. We consider cases depending on which fiber contains $(a_2,b_1)$. See Figure \ref{fig:bad_three_fiber_characterization_case_3} for an illustration.

\begin{figure}
\centering
\begin{tikzpicture}
	\node (3) at (0,0) {
$\begin{array}{c | c c c}
a_0 & 0 & \star & \star \\
a_2 & \star & \star & 2 \\
a_1 & \star & 1 & \star \\
\hline
& b_0 & b_1 & b_2	
\end{array}$};
	\node (3a) at (-5,-2.5) {
$\begin{matrix}
0 & \star & \star \\
\star & \underline{0} & \underline{2} \\
\star & \underline{1} & \underline{\star}
\end{matrix}$};
	\draw[->] (3.south) -- (3a.north);
	\node (3b) at (-3,-2.5) {
$\begin{matrix}
0 & \star & \star \\
\star & 1 & 2 \\
\star & 1 & \star
\end{matrix}$};
	\draw[->] (3.south) -- (3b.north);
	\node (3ba) at (-5,-5) {
$\begin{matrix}
0 & \star & \star \\
\underline{0} & \underline{1} & \underline{2} \\
\star & 1 & \star
\end{matrix}$};
	\draw[->] (3b.south) -- (3ba.north);
	\node (3bb) at (-3,-5) {
$\begin{matrix}
\underline{0} & \star & \underline{\star} \\
\underline{1} & 1 & \underline{2} \\
\star & 1 & \star
\end{matrix}$};
	\draw[->] (3b.south) -- (3bb.north);
	\node (3bc) at (-1,-5) {
$\begin{matrix}
\underline{0} & \underline{\star} & \star \\
\underline{2} & \underline{1} & 2 \\
\star & 1 & \star
\end{matrix}$};
	\draw[->] (3b.south) -- (3bc.north);
	\node (3c) at (3,-2.5) {
$\begin{matrix}
0 & \star & \star \\
\star & 2 & 2 \\
\star & 1 & \star
\end{matrix}$};
	\draw[->] (3.south) -- (3c.north);
	\node (3ca) at (1,-5) {
$\begin{matrix}
0 & \underline{0} & \star \\
\star & \underline{2} & 2 \\
\star & \underline{1} & \star
\end{matrix}$};
	\draw[->] (3c.south) -- (3ca.north);
	\node (3cb) at (3,-5) {
$\begin{matrix}
\underline{0} & \underline{1} & \star \\
\underline{\star} & \underline{2} & 2 \\
\star & 1	 & \star
\end{matrix}$};
	\draw[->] (3c.south) -- (3cb.north);
	\node (3cc) at (5,-5) {
$\begin{matrix}
\underline{0} & \underline{2} & \star \\
\star & 2 & 2 \\
\underline{\star} & \underline{1} & \star
\end{matrix}$};
	\draw[->] (3c.south) -- (3cc.north);
\end{tikzpicture}
\caption{Case 3 in Lemma \ref{lem:bad_three_fiber_characterization}. In the array on the top level, for each $i < 3$, the number $i$ lies in position $(a_i,b_i)$, meaning that $\Psi(a_i,b_i) = i$. On the middle level, we have Case 3a on the left, followed by Cases 3b and 3c. On the bottom level, we have various subcases. For example, in the array on the bottom right, $0$ lies in position $(a_0,b_0)$, $2$ lies in position $(a_0,b_1)$, and $1$ lies in position $(a_1,b_1)$. Together with $(a_1,b_0)$, they form a rectangle satisfying (2).}
\label{fig:bad_three_fiber_characterization_case_3}
\end{figure}
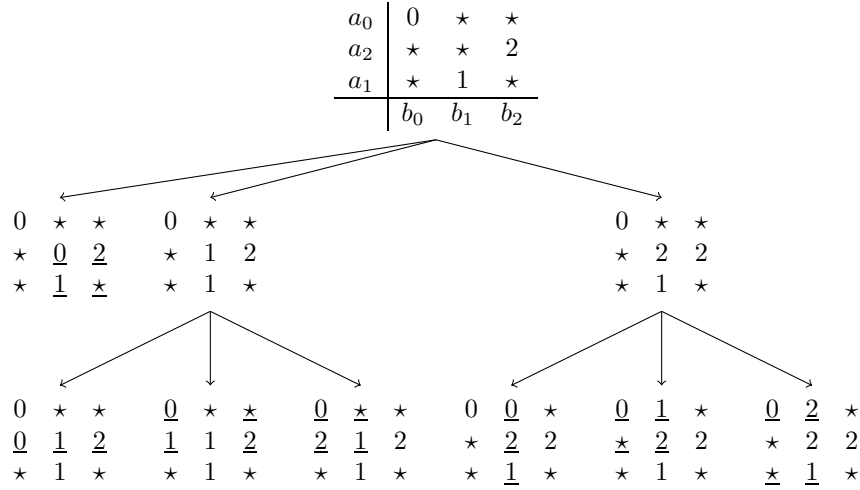

\underline{Case 3a.} $(a_2,b_1)$ and $(a_0,b_0)$ lie in the same fiber. Then we satisfy (2): $(a_1,b_2)$, $(a_1,b_1)$, $(a_2,b_1)$, and $(a_2,b_2)$ form a rectangle with at least one pair from each of the three fibers.

\underline{Case 3b.} $(a_2,b_1)$ and $(a_1,b_1)$ lie in the same fiber. Then we consider cases depending on which fiber contains $(a_2,b_0)$. In all cases, we satisfy either (1) or (2).


\underline{Case 3c.} $(a_2,b_1)$ and $(a_2,b_2)$ lie in the same fiber. We consider cases depending on which fiber contains $(a_0,b_1)$. The argument is symmetric to Case 3b.
\end{proof}

\begin{proposition}\label{44prop}
$\RT^1_8 \nured \RT^1_4 \times \RT^1_4$.
\end{proposition}
\begin{proof}
Towards a contradiction, fix forward functionals $\Phi_0$, $\Phi_1$ and a surjective partial function $\Psi: 4 \times 4 \to 8$ witnessing that $\RTc^1_8 \sured \RTc^1_4 \times \RTc^1_4$. If $\Psi$ has any singletons, we can derive a contradiction using the proof of Proposition \ref{prop:k=1_upper_bound}. Hence we assume that $\Psi$ has no singletons. There are sixteen pairs in $4 \times 4$, so $\Psi$ must be total, and all of the eight fibers in $\Psi$ must contain exactly two pairs each.

As discussed previously, we derive a contradiction by producing a set $S$ that satisfies ($\ast$) and consists of three elements. In other words, we show that there is a collection of three fibers that is not bad. To that end, we give an upper bound for the number of bad collections of three fibers. Since each fiber contains exactly two pairs, it is either contained in a row or column, or lies in diagonal position. Let $k$ be the number of fibers that are contained in some row or column.

First, we give an upper bound for the number of collections that satisfy (2) in Lemma \ref{lem:bad_three_fiber_characterization}. It suffices to give an upper bound for the number of rectangles that intersect at most three fibers. Such rectangles have two possible forms, and we count those cases separately.

\underline{Case 1.} The rectangle contains at least one of those $k$ fibers. There are at most $(4-1)k = 3k$ many such rectangles.

\underline{Case 2.} The rectangle contains at least one fiber in diagonal position. There are at most $8-k$ many such rectangles.

Therefore, there are at most $3k + (8-k) = 2k + 8$ many rectangles that intersect at most three fibers. So there are at most $2k+8$ many collections that satisfy (2).

Next, we give an upper bound for the number of collections that satisfy (1) in Lemma \ref{lem:bad_three_fiber_characterization}.

\underline{Case 1.} If a row/column contains two fibers (and hence nothing else), then said row/column does not contribute to our upper bound. Let $l$ be the number of such rows and columns. Note that $2l \leq k$.

\underline{Case 2.} If a row/column contains one fiber, as well as two other vertices from two different fibers, then said row/column contributes one collection to our upper bound. There are $k-2l$ many such rows/columns.

\underline{Case 3.} Finally, the remaining $8+l-k$ many rows or columns contribute $\binom{4}{3} = 4$ collections each.

Therefore, there are at most
\[ l \cdot 0 + (k-2l) \cdot 1 + (8+l-k) \cdot 4 = 32 - 3k + 2l \leq 32-2k \]
many collections that satisfy (1).

We conclude that there are at most $(2k + 8) + (32 - 2k) = 40$ bad collections of three fibers. There are $\binom{8}{3} = 56 > 40$ collections of three fibers in total, so we can define $S$ to be the image under $\Psi$ of any collection that is not bad. Then $S$ satisfies ($\ast$), which is a contradiction.
\end{proof}
	
\bibliography{Papers}
\bibliographystyle{plain}

\end{document}